\documentclass[12pt]{amsart}
\usepackage{amsmath,amsthm,amssymb,amsfonts}
\usepackage{latexsym}

\theoremstyle{definition}
\newtheorem{defn}{Definition}[section]
\theoremstyle{plain}
\newtheorem{rem}[defn]{Remark}
\theoremstyle{definition}

\theoremstyle{plain}
\newtheorem{lem}[defn]{Lemma}
\newtheorem{thm}[defn]{Theorem}

\theoremstyle{plain}
\newtheorem{prop}[defn]{Proposition}
\theoremstyle{plain}
\newtheorem{cor}[defn]{Corollary}

\newcommand{\inner}[1]{\left\langle #1 \right\rangle}

\renewcommand{\leq}{\leqslant}
\renewcommand{\geq}{\geqslant}

\numberwithin{equation}{section}

\begin{document}
\title{Norm and Numerical Peak Holomorphic Functions on Banach Spaces}
\date{}
\author[S.~G.~Kim]{Sung Guen Kim}
\author[H.~J.~Lee]{Han Ju Lee}

\thanks{The second named author was supported by the Korea Research
Foundation Grant funded by the Korean Government(MOEHRD)
(KRF-2006-352-C00003)}

\baselineskip=.6cm

\begin{abstract}
We introduce the notion of numerical (strong) peak function and
investigate the denseness of the norm and numerical peak functions
on complex Banach spaces. Let $A_b(B_X:X)$ be the Banach space of
all bounded continuous functions $f$ on the unit ball  $B_X$ of a
Banach space $X$ and their restrictions $f|_{B_X^\circ}$ to the open
unit ball are holomorphic. In finite dimensional spaces, we show
that the intersection of the set of all norm peak functions and the
set of all numerical peak functions  is a dense $G_\delta$ subset of
$A_b(B_X:X)$. We also prove that if $X$ is a smooth Banach space
with the Radon-Nikod\'ym property, then the set of all numerical
strong peak functions is dense in $A_b(B_X:X)$. In particular, when
$X=L_p(\mu)$ $(1<p<\infty)$ or $X=\ell_1$, it is shown that the
intersection of the set of all norm strong peak functions and the
set of all numerical strong peak functions is a dense $G_\delta$
subset of $A_b(B_X:X)$.

In the meanwhile, we study the properties of the numerical radius of
an holomorphic function and the numerical index of subspaces of
$A_b(B_X:X)$.  As an application, the existence and properties of
numerical boundary of $A_b(B_X:X)$ are studied. Finally, the
numerical peak function in $A_b(B_X:X)$ is characterized when
$X=\ell_\infty^n$ and some negative results on the denseness of
numerical (strong) peak holomorphic functions are given.
\end{abstract}

\maketitle

\section{Introduction and Preliminaries}

In this paper, we consider only  complex Banach spaces. Given a
Banach space $X$, we denote by $B_X$ and $S_X$ its closed unit ball
and unit sphere, respectively. Let  $X^*$ be the dual space of $X$.
If $X$ and $Y$ are Banach spaces, an {\it $N$-homogeneous
polynomial} $P$ from $X$ to $Y$ is a mapping such that there is an
$N$-linear (bounded) mapping $L$ from $X$ to $Y$ such that $P(x) =
L(x, \dots, x)$ for every  $x$ in $X$. ${\mathcal P}(^N X:Y)$ denote
the Banach space of all $N$-homogeneous polynomials from $X$ to $Y$,
endowed with the polynomial norm $\|P\|=\sup_{x \in B_X}{\|P(x)\|}$.
A mapping $Q:X\to Y$ is a {\it polynomial} if there exist $m$ and
$P_k \in {\mathcal P}(^k X:Y),~k=0, 1, \ldots, m$ such that
$Q=P_0+P_1+\cdots+P_m$. If $P_m\neq 0$, then we say that $Q$ is a
{\it polynomial of degree $m$.} We denote ${\mathcal P}(X:Y)$ the
normed space of all polynomials from $X$ to $Y$, endowed with the
norm $\|Q\|=\sup_{x \in B_X}{\|Q(x)\|}$. We refer to \cite{D} for
background on polynomials. We are mainly interested in the following
spaces. For two Banach spaces $X$, $Y$ and a Hausdorff topological
space $K$,
\begin{eqnarray*}
C_b(K:Y) &:=& \{ f:K\to Y\; :\; f \mbox{ is a bounded continuous
function on } K\},\\
A_b(B_X: Y) &:=& \{ f\in C_b(B_X: Y):
f\mbox{ is holomorphic on }~{B_X^\circ}\}\\
A_u(B_X: Y) &:=& \{ f\in A_b(B_X: Y): f \mbox{ is uniformly
continuous}\},
\end{eqnarray*} where $B_X^\circ$ is the interior of $B_X$. Then
$C_b(K: Y)$ is a Banach space under the sup norm $\|f\|:=\sup\{
\|f(t)\|_Y : t\in K\}$ and both $A_b(B_X:Y)$ and $A_u(B_X:Y)$ are
closed subspaces of $C_b(B_X: Y)$. In case that $Y$ is the complex
scalar field $\mathbb{C}$, we just write $C_b(B_X)$, $A_b(B_X)$ and
$A_u(B_X)$. The closed subspace of $A_u(B_X: Y)$ consisting of all
weakly uniformly continuous functions is denoted by $A_{wu}(B_X:
Y)$. We denote by $A(B_X:X)$ one of $A_b(B_X:X)$, $A_u(B_X:X)$ and
$A_{wu}(B_X:X)$. Notice that if $X$ is finite dimensional,
$A_b(B_X:X)=A_u(B_X:X)=A_{wu}(B_X:X)$.

We denote by $\tau$ the product topology of the set $S_X\times
S_{X^*}$, where the topologies on $S_X$ and $S_{X^*}$ are the norm
topology of $X$ and the weak-$*$ topology of $X^*$, respectively.
The set $\Pi(X): = \{ (x, x^*)~ : ~ \|x\|=\|x^*\|=1=x^*(x)\}$ is a
$\tau$-closed subset of $S_X\times S_{X^*}$. The {\it spatial
numerical range} of $f$ in $C_b(B_X:X)$ is defined by $W(f) =
\{x^*(f(x)) : (x, x^*)\in \Pi(X)\},$ and the {\it numerical radius}
of $f$ is defined by $ v(f) = \sup\{|\lambda| : \lambda\in W(f)\}.$

Let $f$ be an element of $C_b(K:X)$. We say that $f$ {\it attains
its norm} if there is some $t\in K$ such that $\|f\|= \|f(t)\|_X$.
$f$ is said to be a {\it (norm) peak function} at $t$ if there
exists a unique $t\in K$ such that $\|f\|=\|f(t)\|_X$. It is clear
that every (norm) peak function in $C_b(K: X)$ is norm attaining. A
peak function $f$ at $t$ is said to be a {\it (norm) strong peak
function} if whenever there is a sequence $\{t_k\}_{k=1}^\infty$ in
$K$ with $\lim_k \|f(t_k)\|_X = \|f\|$, $\{t_k\}_{k=1}^\infty$
converges to $t$. It is easy to see that if $K$ is compact, then
every peak function is a strong peak function. Given a subspace $H$
of $C_b(K)$, we denote by $\rho H$ the set of all points $t\in K$
such that there is a strong peak function $f$ in $H$ with
$\|f\|=|f(t)|$.

Similarly we introduce the notion of numerical peak functions. Let
$f$ be an element of $C_b(B_X:X)$, where $X$ is a Banach space. If
there is some $(x, x^*)$ in $\Pi(X)$ such that $v(f) = |x^*(f(x))|$,
we say that $f$ {\it attains its numerical radius}. $f$ is said to
be a {\em numerical peak function} at $(x, x^*)$ if there exist a
unique $(x, x^*)\in \Pi(X)$ such that $v(f)= |x^*(f(x))|$. In this
case, $(x, x^*)$ is said to be the {\em numerical peak point} of
$f$. The numerical peak function $f$ at $(x, x^*)$ is called a {\it
numerical strong peak function} if whenever there is a sequence
$\{(x_k, x_k^*)\}_{k=1}^\infty$ in $\Pi(X)$ such that $\lim_k
|x^*_k(f(x_k))| = v(f)$, then $\{(x_k, x_k^*)\}_{k=1}^\infty$
converges to $(x, x^*)$ in $\tau$-topology. In this case, $(x, x^*)$
is said to be the {\em numerical strong peak point} of $f$. We say
that a numerical strong peak function $f$ at $(x, x^*)$ is said to
be a {\it very strong numerical peak function} if whenever there is
a sequence $\{(x_k, x_k^*)\}_{k=1}^\infty$ in $\Pi(X)$ satisfying
$\lim_n |x^*_k(f(x_k))|= v(f)$, we get $\lim_k x_k=x$ and $\lim_k
x_k^* =x^*$ in the norm topology.  If $X$ is finite dimensional,
then every numerical peak function is a very strong numerical peak
function.

In 1996, Y.S. Choi and the first named author \cite{CK} initiated
the study of denseness of norm or numerical radius attaining
nonlinear functions, especially, {\em homogeneous polynomials} on a
Banach space. Using the perturbed optimization theorem of Bourgain
\cite{Bo} and Stegall \cite{S}, they proved that if $X$ has the
Radon-Nikod\'ym property, then the set of all norm attaining
functions in ${\mathcal P}(^kX)$ is norm-dense. Concerning the
numerical radius, it was also shown that if $X$ has the
Radon-Nikod\'ym property, then the set of all numerical radius
attaining functions in ${\mathcal P}(^kX:X)$ is norm-dense. M. D.
Acosta, J. Alaminos, D. Garc\'ia and M. Maestre \cite{AAGM} proved
that if $X$ has the Radon-Nikod\'ym property, then the set of all
norm attaining functions in ${A}_{b}(B_X)$ is norm-dense. Recently,
it was shown \cite{CLS} that if $X$ has the Radon-Nikod\'ym
property, the set of all (norm) strong peak functions in $A_b(B_X)$
is dense. Concerning the numerical radius, M. D. Acosta and the
first named author \cite{AK} showed that the set of all numerical
radius attaining functions in $A_b(B_X: X)$ is dense if $X$ has the
Radon-Nikod\'ym property. In this paper, we extend the results of
the above (\cite{CK}, \cite{AAGM}, \cite{CLS}, \cite{AK}) to the
denseness of norm or numerical (strong) peak functions in $A(B_X:
X)$ if $X$ has the Radon-Nikod\'ym property.

Let's briefly sketch the content of this paper. In section 2, we
show that if $X$ is a finite dimensional Banach space, then the set
of all norm and numerical strong peak functions in $A(B_X: X)$ is a
dense $G_\delta$-subset of $A(B_X: X)$. For the extension from the
finite dimensional space to the infinite dimensional space by
approximation, we introduce the following notions. A Banach space
$X$ has $(FPA)$-property with $\{\pi_i, F_i\}_{i\in I}$ if

\begin{enumerate}
\item each $\pi_i$ is a norm-one projection with finite dimensional
range $F_i$,

\item given $\epsilon>0$, for every finite-rank operator $T:X\to F$
for some Banach space $F$ and for every finite dimensional subspace
$G$ of $X$, there is $\pi_i$ such that
\[\|T-T\pi_i\|\le \epsilon,  \ \ \ \  \|I_G - \pi_i|_G\|\le
\epsilon.\] \end{enumerate} As examples, we show that $X$ has
$(FPA)$-property if at least one of the following conditions are
satisfied:

(a) It has a shrinking and monotone finite-dimensional
decomposition.

(b) $X=L_p(\mu)$, where $\mu$ is a finite measure and $1\le p<
\infty$.

We show that if $X$ has $(FPA)$-property, then the set of all
polynomials $Q\in {\mathcal P}(X:X)$ such that there exist a finite
dimensional subspace $F$ and norm-one projection $\pi:X\to F$ such
that $\pi \circ Q\circ \pi = Q$ and $Q|_F$ is a norm and numerical
peak function as a mapping from $B_F$ to $F$ is dense in
$A_{wu}(B_X:X)$.

A subset $\Gamma$ of $\Pi(X)$ is called a {\it numerical boundary}
for a subspace $H$ of $C_b(B_X: X)$ if $ v(f) = \sup\{ |x^*(f(x))|:
(x, x^*)\in \Gamma\}$ for every $f$ in $H$. The projections
$\{\pi_i, F_i\}_{i\in I}$ are said to be {\it parallel} to a
numerical boundary $\Gamma$ of $H$ if each $\pi_i$ has the image
$F_i$ and
\[|\inner{x^*|_{F_i}, \pi_i(x)}| = \|x^*|_{F_i}\|\cdot \|\pi_i(x)\|,\ \ \ \ \forall (x, x^*)\in
\Gamma, \ \ \forall i\in I.\] A projection $\pi:X\to X$ is said to
be {\it strong} if whenever
 $\{\pi(x_k)\}_{k=1}^{\infty}$ is norm-convergent to $y\in S_X$ for
some $\{x_k\}_{k=1}^{\infty}$ in $B_X$, $\{x_k\}_{k=1}^{\infty}$ is
norm-convergent to $y$.

Recall that a Banach space $X$ is said to be {\it locally uniformly
convex} if $x\in S_X$ and there is a sequence $\{x_n\}$ in $B_X$
satisfying $\lim_n \|x_n + x\|=2$, then $\lim_n \|x_n - x\|=0$.
Notice that if $X$ is locally uniformly convex, then every norm-one
projection is strong. We prove that if a smooth Banach space $X$ has
$(FPA)$-property and the corresponding projections are strong and
parallel to $\Pi(X)$, then the set of all norm and numerical strong
peak functions in $A_{wu}(B_X:X)$ is dense. We also prove that if a
Banach space $X$ has $(FPA)$-property with $\{(\pi_i, F_i)\}_{i\in
I}$, the corresponding projections are strong, parallel to $\Pi(X)$,
and $\pi^*_i:X^* \to X^*$'s are also strong, then the set of all
very strong numerical and norm strong peak functions is dense in
$A_{wu}(B_X: X)$.

Let $K$ be a convex subset of a Banach space $X$. An element $x$ in
$K$ is said to be a {\it strongly exposed point} of $K$ if there is
nonzero $x^*\in B_{X^*}$ such that ${\rm Re\ }x^*(x)= \sup \{ {\rm
Re\ }x^*(y):y\in K\}$ and whenever $\lim_n {\rm Re\ } x^*(x_n)= {\rm
Re\ }x^*(x)$ for some sequence $\{x_n\}_{n=1}^\infty$ in $K$, we get
$\lim_n \|x_n - x\|=0$. A Banach space $X$ is said to have the {\it
Radon-Nikod\'ym property} if every nonempty bounded closed convex
subset in $X$ is a closed convex hull of its strongly exposed points
\cite{DU}. The point $x\in B_X$ is said to be a {\it smooth point}
if there is a unique $x^*\in B_{X^*}$ such that ${\rm Re} x^*(x)=1$.
We denote by ${\rm sm} (B_X)$ the set of all smooth points of $B_X$.
We say that a Banach space is {\it smooth} if ${\rm sm}(B_X)$ is the
unit sphere $S_X$.

When $X$ is a smooth Banach space with the Radon-Nikod\'ym property,
it is shown that the set of all numerical strong peak functions is
dense in $A(B_X: X)$. In particular, if $X$ is a Banach space with
the Radon-Nikod\'ym property and $X^*$ is locally uniformly convex,
then the set of all norm and numerical strong peak functions in
$A(B_X: X)$ is a dense $G_\delta$-subset of $A(B_X: X)$. As a
corollary, if $1<p<\infty$ and $X= L_p(\mu)$ for a measure space
$\mu$, then the set of all norm and numerical strong peak functions
in $A(B_X: X)$ is a dense $G_\delta$-subset of $A(B_X:X)$. In this
case, every numerical strong peak function is a very strong
numerical peak function. We also prove that the set of all norm and
numerical strong peak functions in $A(B_{l_1}: l_1)$ is a dense
$G_\delta$-subset of $A(B_{l_1}: l_1)$.

Concerning the numerical index of subspaces of $A_b(B_X:X)$, we
extend the recent result of E. Ed-dari \cite{Ed}. Although it is not
directly related to the denseness of numerical peak holomorphic
functions, it is a byproduct of the study. Let $H$ be a subspace of
$A_b(B_X:X)$. We introduce the ($H$-) numerical index by $N(H) =
\inf \{ v(f): f\in H, \|f\|=1\}$. When $H=\mathcal{P}({\ }^k X:X)$
for some $k\ge 1$, the polynomial numerical index $N(H)$ is usually
denoted by $n^{(k)}(X)$ (see \cite{CGKM}).

For norm-one projection $\pi$ with range $F$ and for any subspace
$H$ of $A_b(B_X: X)$, define $H_F = \{ \pi\circ f\circ \pi
|_F:B_F\to F: f\in H\}.$ We prove that if $X$  has $(FPA)$-property
with $\{(\pi_i, F_i)\}_{i\in I}$ and the corresponding projections
are parallel to a numerical boundary of a subspace $H$, then $N(H)=
\inf_{i\in I} N(H_{F_i})$. In fact, $N(H)$ is a decreasing limit of
the right-hand side  with respect to the inclusion partial order. As
a corollary we also extended Ed-dari's result to the polynomial
numerical indices of $l_p$. In fact, the first named author
\cite{K2} extended Ed-dari's result(\cite{Ed}, Theorem 2.1) to the
polynomial numerical indices of (real or complex) $l_p$ of order $k$
as follows: Let $1<p<\infty$ and $k\in \mathbb{N}$ be fixed. Then
$n^{(k)}(l_p)=\inf\{n^{(k)}(l_p^m): m \in \mathbb{N}\}$ and the
sequence $\{n^{(k)}(l_p^m)\}_{m \in \mathbb{N}}$ is decreasing.

In section 3, we give some applications of the denseness of
numerical strong peak holomorphic functions. More precisely, we show
that if the set of numerical strong peak functions are dense in a
subspace $A$ of $C_b(B_X:X)$ then the numerical Shilov boundary of
$A$ exists and it is the $\tau$-closure of the set of all numerical
strong peak points. On the other hand, using the Lindensrauss method
\cite{Li}, we show that if there is a numerical boundary $\Gamma$ of
$A(B_X:X)$ such that the first component of every element in
$\Gamma$ is a strong peak point of $A(B_X)$, then the set of all
numerical radius attaining elements in $A(B_X:X)$ is dense. As
corollaries, it is shown that if $X$ is either a locally uniformly
convex Banach space or an order continuous sequence space with local
uniform $c$-convexity, then the set of all numerical radius
attaining elements is dense in $A(B_X:X)$. Recently, the second
named author shows \cite{Lee5} that if $\Pi(X)$ is metrizable and
the set $\Gamma= \{ (x, x^*)\in \Pi(X) : x\in \rho {A}(B_X)\cap {\rm
sm}(B_X)\}$ is a numerical boundary of $A(B_X:X)$, then the set of
all numerical strong peak functions is dense in $A(B_X:X)$.

In section 4,  we characterize the numerical peak function in
$A_b(B_X:X)$ when $X=\ell_\infty^n$. More precisely, setting
$X=\ell_\infty^n$, an element $f$ in $A_b(B_X:X)$ is a numerical
peak function in $A_b(B_X:X)$ if and only if there exist unique
$x_0\in extB_{X}$ and $1\le m_0\le n$ such that
\begin{enumerate}
\item[(a)] $v(f)=\|f\|=\|f(x_0)\|>\|f(x)\|$ for every $x\in B_{X}$ with
$x\neq x_0$;

\item[(b)] $v(f)=\|f\|=|f(x_0)(m_0)|>|f(x_0)(m)|$ for every $1\le m\le
n$ with $m\neq m_0$.
\end{enumerate}
For negative results for denseness of numerical peak holomorphic
functions on a classical Banach space, we prove the following:
\begin{enumerate}
\item[(1)] Let $X$ be a complex Banach space having $(\beta)$-and $
{\bf Q}$-properties with $\rho=0$. There are no numerical peak
functions in $A_b(B_{X}:X)$.

\item[(2)] Let $\Omega$ be a locally compact Hausdorff space with more
than 2 elements. Let $X=C_0(\Omega)$. Then there are no numerical
peak functions in $A_b(B_X:X)$.

\item[(3)] Let $K$ be an infinite compact Hausdorff space. Then there
are no numerical strong peak functions in $A_b(B_{C(K)}:C(K))$.
Neither are there in $A_{wu}(B_{L_1[0,1]}: L_1[0,1])$.
\end{enumerate}

\section{Denseness of Numerical Peak Holomorphic Functions}

Let $K$ be a Hausdorff space and $Y$ be a complex Banach space.
Consider the product space  $K\times B_{Y^*}$ where $B_{Y^*}$ is
equipped with the weak-$*$ topology. Given a subspace $A$ of
$C_b(K:Y)$, consider the map $\varphi: f\in A \mapsto \tilde{f}\in
C_b(K\times B_{Y^*})$ defined by \[\tilde{f}(x, y^*) = y^*(f(x)), \
\ \forall (x, y^*)\in K\times B_{Y^*}.\] Then $\varphi$ is a linear
isometry, and its image $\tilde{A}$ of $A$ is also a subspace of
$C_b(K\times B_{Y^*})$.  We say that the subspace $A$ of $C_b(K: Y)$
is {\it separating} if the following conditions hold:

\begin{enumerate}
\item[(i)] If $x\neq y$ in $K$, then $\delta_{(x,x^*)} \neq
\delta_{(y,y^*)}$ on $\tilde{A}$ for every $x^*, y^* \in S_{Y^*}$.

\item[(ii)] Given $x\in K$ with $\delta_x \neq 0$ on $A$, we have
$\delta_{(x,x^*)} \neq \delta_{(x,y^*)}$ on $\tilde{A}$ for every
$x^*\neq y^*$ in $ext(B_{Y^*})$,
\end{enumerate}
where $\delta_t$ (for some $t\in K$) is a linear map from $C_b(K:Y)$
defined by $\delta_t(f) = f(t)$.

We need the theorem in \cite{CLS}.

\begin{thm}\cite{CLS}\label{thm:cls}
Let $Y$ be a Banach space and let $A$ be a nontrivial separating
separable subspace of $C(K: Y)$ on a compact Hausdorff space $K$.
Then the set $\{ f\in A:f \mbox{ is a peak function at some } t\in
K, f(t)/\|f\| \in {\rm sm}(B_Y) \}$ is a dense $G_\delta$-subset of
$A$.
\end{thm}

\begin{prop}\label{prop:densefinitenumerical}Let $X$ be a finite dimensional Banach space.
Suppose that a subspace $H$ of $C_b(B_X: X)$ contains the functions
of the form  \begin{equation}\label{eq:functionform} 1 \otimes x,\ \
\ \ \ y^*\otimes z, \ \ \ \ \forall x, z\in X,\ \ \ \forall y^*\in
X^*.\end{equation} If the numerical index $N(H)>0$, then the set of
all numerical peak functions in $H$ is a dense $G_\delta$-subset of
$H$.
\end{prop}
\begin{proof}
Suppose that $N(H)>0$. Then $N(H)\|f\|\le v(f)\le \|f\|$ for all
$f\in H$. So the $v(\cdot)$ is a complete norm on $H$.

Consider the linear map $f\mapsto \tilde f$ from $H$ into
$C(\Pi(X))$ defined by \[ \tilde f(x, x^*) = x^*(f(x)).\] Notice
that $v(f) = \|\tilde f\|$ for every $f\in H$. Let $\tilde{ H}$ be
the image in $C(\Pi(X))$. So two Banach spaces $(H, v)$ and
$(\tilde{H}, \|\cdot \|)$ is isometrically isomorphic.

Since $X$ is finite dimensional, $\Pi(X)$ is compact metrizable so
$C(\Pi(X))$ is separable. Then $\tilde{H}$ is a separable subspace
of $C(\Pi(X))$.

Claim: $\tilde{H}$ is separating

Let $(x, x^*) \neq (y, y^*)\in \Pi(X)$ and let $\alpha, \beta\in
S_\mathbb{C}$. If $\alpha x^* \neq \beta y^*$, then choose $z\in
S_X$ such that $\alpha x^*(z) \neq \beta y^*(z)$. Set $f:= 1\otimes
z\in H$. Then
\[\alpha \delta_{(x,x^*)}(\tilde f) =\alpha \tilde f (x,x^*) =
\alpha x^*(z) \neq \beta y^*(z) = \beta \tilde f(y,
y^*)=\beta\delta_{(y,y^*)}(\tilde f).\] If $\alpha x^* = \beta y^*$,
then $x\neq y$, and choose $z^*\in S_{X^*}$ such that $z^*(x) \neq
z^*(y)$. Set $g: = z^* \otimes x \in H$. Then $\beta y^*(x) = \alpha
\neq 0$ and
\[ \alpha \tilde g (x, x^*) = \alpha z^*(x)x^*(x) =\beta z^*(x)y^*(x)\neq  \beta z^*(y) y^*(x)=
\beta \tilde g(y, y^*),\] hence $\alpha \delta_{(t,t^*)}(\tilde
g)\neq \beta\delta_{(s,s^*)}(\tilde g)$. Therefore $\tilde H$ is a
separating separable subspace of $C(\Pi(X))$. By
Theorem~\ref{thm:cls}, the set of peak functions in $\tilde{H}$ is
dense. So we get the desired result.
\end{proof}

Recall the following theorem of L.A. Harris \cite{H}.

\begin{thm}[Harris]\label{thm:harris}
Let $h\in A_b(B_X:X)$ and  $P_m$ the $m$-th term of the Taylor
series expansion for $h$ about 0. Then $\|P_m\|\le k_m v(h),$ where
$k_0=1$, $k_1=e$ and $k_m = m^{m/(m-1)}$ for $m\ge 2$.
\end{thm}

\begin{prop}\label{prop:positivenumerical1}
Let $m\ge 1$ be a natural number and $H_m$ be the subspace of
$A_b(B_X: X)$ consisting of all polynomials of degree $\le m$. Then
its numerical index $N(H_m)$ is positive.
\end{prop}
\begin{proof}
Let $h\in H_m$ and $x\in B_X^\circ$. Then  $h(x) =\sum_{k=0}^m
P_k(x)$ Then by Theorem~\ref{thm:harris},
\[ \sum_{k=0}^m \|P_k(x)\| \le \sum_{k=0}^m \|P_k\| \le
\sum_{k=0}^m k_m v(h)\le c_m v(h),\] where $c_m=\sum_{k=0}^m k_m>0$.
Hence $\|h\| \le c_m v(h).$ Therefore $N(H_m)\ge c_m^{-1}>0$.
\end{proof}

From Proposition~\ref{prop:densefinitenumerical} and
\ref{prop:positivenumerical1}, we have the following.

\begin{prop}\label{prop:pro1}
Let $m\ge 1$ be a natural number and $H_m$ be the subspace of
$A_b(B_X: X)$ consisting of all polynomials of degree $\le m$. If
$X$ is finite dimensional, the set of all numerical peak functions
in $H_m$ is a dense $G_\delta$-subset of $H_m$.
\end{prop}

\begin{thm}\label{prop:peak}
Let $X$ be a finite dimensional complex Banach space. Then the set
of all norm and numerical peak functions in $A_u(B_X: X)$ is dense.
In fact, setting \begin{align*}\Delta_1= \{ f\in A_u(B_X:X)&:f
\mbox{ is peak function at~ } t\in B_X \mbox{ and } \\ &\ \ \ \ \ \
\ \ \ f(t)/\|f\| \mbox{ is a smooth point of } B_X\},
\\
\Delta_2= \{ f\in A_u(B_X: X)&:f \mbox{ is a numerical peak
function} \},\end{align*} the intersection $\Delta_1\cap \Delta_2$
is dense in $A_u(B_X:X)$.
\end{thm}
\begin{proof} Notice that if $X$ is a finite dimensional Banach space, then the
subspace $H_m\subset A_u(B_X: X)$ of all polynomials of degree $\le
m$ for some $m\ge 1$ is a separating subspace of  $C(B_X: X)$.
Hence, by Theorem~\ref{thm:cls} and Proposition~\ref{prop:pro1}, the
intersection $\Delta_1 \cap \Delta_2\cap H_m$ is a dense
$G_\delta$-subset of $H_m$. So we get the following theorem.

Let $\epsilon>0$ and $f\in A_u(B_X:X)$. Then choose $p_m$ be a
polynomial of degree $m$ such that $\|f-p_m\|\le \epsilon$. So there
is a norm and  numerical peak function $q\in \Delta_1\cap
\Delta_2\cap H_m$ such that $\|q-p_m\|\le \epsilon$. Hence
$\|f-q\|\le 2\epsilon$. This completes the proof.
\end{proof}

\begin{rem}
Proposition~\ref{prop:Gdelta} shows that the intersection
$\Delta_1\cap \Delta_2$ is in fact a dense $G_\delta$ subset of
$A_u(B_X:X)$ if $X$ is finite dimensional.
\end{rem}

We will say that the $k$-linear mapping $L:X\times \dots \times X\to
Y$ is {\it of finite-type} if it can be written as
\[ L(x_1, \dots, x_k ) = \sum_{i=1}^m x_{1,i}^*(x_1) \dots
x_{k,i}^*(x_k) y_i,~~\forall x_1, \ldots, x_k\in X\] for some
$m\in\mathbb{N}$, $x_{1,1}^*, \dots, x_{k,m}^*$ in $X^*$ and $y_1,
\dots, y_m$ in $Y$.
 We will denote by $L_f(^kX: Y)$ the space of the
$k$-linear mapping from $X$ to $Y$ of finite type. If $P$ is
associated to such a $k$-linear mapping, we will say that it is a
{\it finite-type polynomial}.

\begin{prop}\label{prop:3peak1}
Suppose that $X$ has $(FPA)$-property with $\{(\pi_i, F_i)\}_{i\in
F}$. Then the set of all polynomials $Q\in {\mathcal P}(X:X)$ such
that there exist a projection $\pi_i:X\to F_i$ such that $\pi_i
\circ Q\circ \pi_i = Q$ and $Q|_{F_i}$ is a norm and numerical peak
function as a mapping from $B_{F_i}$ to $F_i$ is dense in
$A_{wu}(B_X: X)$.
\end{prop}

\begin{proof} We follow the ideas in \cite{AAGM}.
The subset of continuous polynomials is always dense in $A_u(B_X:
X)$. Given $f\in A_u(B_X: X)$ and $n\in \mathbb{N}$, it is the limit
in $A_u(B_X: X)$ of sequence of functions $\{f_n \}_n$ defined by
$f_N(x): = f(\frac{n}{n+1}x)$. Then $f_n$ belongs to
$A_b(\frac{n+1}n B_X: X)$.  Thus the Taylor series expansion of
$f_n$ at 0 converges uniformly on $B_X$ for all $n$.

We will also use fact that if $\sum_{k=0}^\infty P_k$ is the Taylor
series expansion of $f\in A_{wu}(B_X: X)$ at 0, then $P_k$ is weakly
uniformly continuous on $B_X$ for all $k$.

Since $X$ has $(FPA)$-property, $X^*$ has the approximation property
(see \cite[Lemma 3.1]{JW}). Then the subspace of $k$-homogeneous
polynomials of finite-type restricted on $B_X$ is dense in the
subspace of all $k$-homogeneous polynomials which are weakly
uniformly continuous on $B_X$ (see \cite[Proposition 2.8]{D}). Thus
the subspace of the polynomials of finite-type restricted to the
closed unit ball of $X$ is dense in $A_{wu}(B_X: X)$.

Assume that $P$ is a finite-type polynomial that can be written as a
finite sum $P=\sum_{k=0}^n P_k$, where each $P_k$ is an homogeneous
finite-type polynomial with degree $k$. Consider the symmetric
$k$-linear form $A_k$ associated to the corresponding polynomial
$P_k$. Since $P_k$ is a finite-type polynomial, then $T_k:X\to
L_f(^{k-1}X: X)$ given by
\[ T_k(x)(x_1, \dots, x_{k-1}): = A_k(x, x_1, \dots, x_{k-1}),~~\forall x\in
X\] is a linear finite-rank operator for any $1\le k\le n$.

The direct sum of these operators, that is, the operator
\[ T:X\to \bigoplus_{k=1}^n L_f(^{k-1}X: X)\] given by
$ T(x): = (T_1(x), \dots, T_N(x)),~~\forall x\in X$ is also of
finite rank.

By the assumption on $X$, given any $\epsilon>0$, there is a
norm-one projection $\pi:=\pi_i:X\to X$ with a finite-dimensional
range such that $\|T-T\pi\|\le \epsilon$ and $\|\pi|_G-I_G\|\le
\epsilon$, where $G$ is the span of $\bigcup_{k=1}^n P_k(X)$.

Let $B_k$ be the symmetric $k$-linear mapping given by $B_k: = A_k
\circ (\pi, \dots, \pi)$, and let $Q_k$ the associated polynomial.
It happens that $Q_k=P_k\circ \pi$. Now for $\|x\|\le 1$, we have

\begin{eqnarray*} \|P_k\circ \pi
(x) &-& P_k(x)\| = \left\|\sum_{j=0}^{k-1} {k\choose j}
A_k((x-\pi(x))^{k-j}, \pi(x)^j) \right\|\\
&=&\left\|\sum_{j=0}^{k-1} {k\choose j} (T_k - T_k\circ
\pi)(x)((x-\pi(x))^{k-j-1}, \pi(x)^j) \right\|\\
&\le& \sum_{j=0}^{k-1} {k\choose j} \|T_k-T_k\circ
\pi\|\|x\|\|x-\pi(x)\|^{k-j-1}\|\pi(x)\|^j\\
&\le& \epsilon \sum_{j=0}^{k-1} {k\choose j}
 2^{k-j-1}\le  4^k \epsilon.
\end{eqnarray*}
Then $\|P_k\circ \pi - P_k\|\le 4^k\epsilon$ and
\[ \|\pi\circ P_k \circ \pi - P_k \|\le  \|\pi\circ P_k \circ \pi - \pi\circ
P_k\| + \|\pi\circ P_k - P_k\|\le 2\cdot 4^k\epsilon.\] Let $R_k =
\pi \circ P_k \circ \pi$ and $R= P_0 + \sum_{k=1}^n R_k$. Then
$\|R-P\|\le 2n4^{n}\epsilon$. By Theorem~\ref{prop:peak}, there is a
numerical and norm peak polynomial $Q':\pi(X)\to \pi(X)$ of degree
$\le n$ such that $\|R|_{\pi(X)} - Q'\|\le \epsilon$. Setting $Q:=
Q' \circ \pi$, $\|P-Q\|\le (2n4^{n}+2)\epsilon$. The proof is done.
\end{proof}

Following \cite[Definition 1.g.1]{LT}, a Banach space $X$ has a {\it
finite-dimensional Schauder decomposition} (FDD for short) if there
is a sequence $\{X_n\}$ of finite-dimensional spaces such that every
$x\in X$ has a unique representation of the form
$x=\sum_{n=1}^\infty x_n$, where $x_n\in X_n$ for every $n$. In such
a case, the projections given by $P_N(x) = \sum_{i=1}^n x_i$, are
linear and bounded operators. If moreover, for every $x^*\in X^*$,
it is satisfied that $\|P_n^*x^* - x^*\|\to 0$, the FDD is called
{\it shrinking}. The FDD is said to be {\it monotone} if $\|P_n\|=1$
for every $n$.

\begin{cor}\label{cor:3condi}
Assume that $X$ is a complex Banach space satisfying at least one of
the following conditions:
\begin{enumerate}
\item It has a shrinking and monotone finite-dimensional
decomposition.

\item $X=L_p(\mu)$, where $\mu$ is a finite measure and $1\le p<
\infty$.
\end{enumerate} Then the set of all
polynomials $Q\in {\mathcal P}(X:X)$ such that there exist a finite
dimensional subspace $F$ and norm-one projection $\pi:X\to F$ such
that $\pi \circ Q\circ \pi = Q$ and $Q|_F$ is a peak and numerical
peak function as a mapping from $B_F$ to $F$ is dense in
$A_{wu}(B_X: X)$.
\end{cor}

\begin{proof} By Proposition~\ref{prop:3peak1}, we need show that
the spaces satisfying condition (1) and (2) have $(FPA)$-property.
Let $T:X\to F$ be a linear operator from $X$ to a finite dimensional
space $F$ and $G$ be a finite dimensional subspace $G$ of $X$. Given
$\epsilon>0$, there is an $\epsilon/3$-net $\{g_1, \dots, g_n\}$ in
$B_G$ and $T$ can be written as $\sum_{i=1}^m x_i^*\otimes  y_i$ for
some $x_1^*,\dots, x_m^*\in X^*$ and $y_1,\dots, y_m\in F$.

(1) Suppose that $X$ has a shrinking monotone finite-dimensional
decomposition. Then there is $N\in \mathbb{N}$ such that
\[\max_{1\le i\le m}\|y_i\| \cdot \sum_{i=1}^m \|P_N^* x^*_i -  x_i^*\| \le
\epsilon,\ \  \ \ \max_{1\le j\le n} \|P_N g_j - g_j \|\le
\epsilon/3.\] Then for any $x\in B_X$,
\begin{eqnarray*}
\|TP_N x - Tx\| &= \|\sum_{i=1}^m (P_N^*x_i^*)(x) y_i - \sum_{i=1}^m
x_i^*(x)y_i\|\\
&\le \max_{1\le j\le n}\|y_i\| \cdot \sum_{i=1}^m \|P_N^* x^*_i -
x_i^*\|\le \epsilon,
\end{eqnarray*} hence $\|TP_N -  T\|\le \epsilon$. For any $x\in
B_G$, there is $g_j$ such that $\|x-g_j\|\le \epsilon/3$, then
because the decomposition is monotone,
\begin{eqnarray*}
\|P_N x - x\| &\le \|P_N(x-g_j)\| + \|P_N g_j - g_j\| +\|x-g_j\|\\
&\le 2\|x-g_j \| + \|P_N g_j - g_j \| \le \epsilon.
\end{eqnarray*} So taking $P=P_N$, we obtained the desired result.

(2) Suppose that $X=L^p(\mu)$. We may assume that $\mu$ is a
probability measure. For each $1\le i\le m$, there is $s_i\in
L_q(\mu)$ such that  $1/p+1/q=1$ and $x_i^*(f) = \int f s_i\, d\mu$
($f\in L_p(\mu)$). Then there is a sub-$\sigma$-algebra
$\mathcal{F}$ generated by finite disjoint subsets such that
\[ \max_{1\le j\le n}\|y_i\| \cdot \sum_{i=1}^m \|
E(s_i|\mathcal{F})-s_i\|_q \le \frac{\epsilon}2 , \ \ \ \ \max_{1\le
i\le n}\|E(g_i|\mathcal{F})- g_i\|_p \le \frac\epsilon3\]

Define a projection $P:X\to X$ as $Pf = E(f|\mathcal{F}).$ It is
clear that $P$ is a norm-one projection. For any $f\in B_X$,
\begin{eqnarray*}
\|TPf - Tf\| &=& \|\sum_{i=1}^m (x_i^*)(Pf) y_i - \sum_{i=1}^m
x_i^*(f)y_i\|\\&\le& \max_{1\le j\le n}\|y_i\| \cdot \sum_{i=1}^m
|x^*_i(Pf) -
x_i^*(f)|\\
&\le& \max_{1\le j\le n}\|y_i\| \cdot \sum_{i=1}^m |\int_K
(E(f|\mathcal{F})-f) E(s_i|\mathcal{F}) \, d\mu| \\
&\ &\ \ +\max_{1\le j\le n}\|y_i\| \cdot \sum_{i=1}^m |\int_K
(E(f|\mathcal{F})-f) (E(s_i|\mathcal{F})-s_i) \, d\mu|
\\ &=& 0+ \max_{1\le j\le n}\|y_i\| \cdot \sum_{i=1}^m |\int_K
(E(f|\mathcal{F})-f) (E(s_i|\mathcal{F})-s_i) \, d\mu|\\
&\le& \max_{1\le j\le n}\|y_i\| \cdot 2\sum_{i=1}^m \|f\|_p \|
E(s_i|\mathcal{F})-s_i\|_q \le \epsilon.
\end{eqnarray*} On the other hand, for any $f\in
B_G$, there is $g_j$ such that $\|f-g_j\|\le \epsilon/3$. So
\begin{eqnarray*}
\|Pf - f\| &\le \|P(f-g_j)\| + \|Pg_j - g_j\| +\|x-g_j\|\\
&\le 2\|f-g_j \| + \|P g_j - g_j \| \le \epsilon.
\end{eqnarray*} We obtained the desired result. The proof is
complete.
\end{proof}

\begin{rem}
If $X$ is a Banach space satisfying $(FPA)$-property, then the set
of polynomials in $B_{A_{wu}(B_X:X)}$ which has a nontrivial
invariant subspace and has a fixed point is dense in
$B_{A_{wu}(B_X:X)}$.
\end{rem}

\begin{prop}\label{prop:sameradius1}
Suppose that the Banach space $X$ has the $(FPA)$-property with
$\{\pi_i, F_i\}_{i\in I}$. Then $N(H)\ge \inf_{i\in I} N(H_{F_i})$.
\end{prop}

\begin{proof}
Let  $f\in S_H$. given $\epsilon>0$, there is a norm one projection
$\pi$ with a finite dimensional range $F$ such that $\|\pi\circ f
\circ \pi \|\ge 1-\epsilon$. Let $g = \pi\circ f \circ \pi |_F$ as a
map in $H_F$.
\[ v_F(g) \ge N(H_F) \|g\|\ge N(H_F)(1-\epsilon).\]
Then there is $(y, y^*)\in \Pi(H_F)$ such that $ v_{H_F}(g) =
|y^*(g(y))|$ since $F$ is finite dimensional. Notice that $(y,
\pi^*(y^*))\in \Pi(X)$ and so \[ v_F(g) = |\pi^*y^*(f(\pi(y)))|=
|\pi^*x^*(f(y))|\le v_H(f).\] Hence $v_H(f) \ge (1-\epsilon)
N(H_F)\ge (1-\epsilon)\inf_{i\in I} N(H_{F_i})$. Therefore $N(H)\ge
\inf_{i\in I} N(H_{F_i})$.
\end{proof}

\begin{lem}\label{lem:basic} Let $X$ be a  Banach space and
 $f\in A_b(B_X: X)$. Suppose that there is $y$ in $B_X$ and $y^*\in
B_{X^*}$ such that $|y^*(y)|= \|y^*\| \cdot \|y\|$. Then $
|y^*(f(y))|\le v(f).$ In particular, $\|f(0)\|\le v(f)$.
\end{lem}
\begin{proof}If $y^*=0$, then it is clear. So we may assume that $y^*\neq 0$.
 Suppose first that $y=0$. By the Bishop-Phelps theorem \cite{BP}, given $\epsilon>0$, there is $w^*\in
B_{X^*}\setminus \{0\}$ such that $\|w^*- y^*\|\le \epsilon$ and
$w^*$ attains its norm at some $x\in S_X$. Then by the maximum
modulus theorem,
\begin{eqnarray*}  |y^*(f(0))| &\le |w^*(f(0))| + \epsilon \|f(0)\| \le |\frac{w^*}{\|w^*\|} (f(0))| + \epsilon\|f(0)\|\\
&\le \max_{|\lambda|= 1} |\frac{w^*}{\|w^*\|} (f(\lambda x))| +
\epsilon \|f(0)\| \le v(f)+ \epsilon\|f(0)\|.\end{eqnarray*} Since
$\epsilon>0$ is arbitrary, $|y^*(f(0))|\le v(f)$.

In case that $y\neq 0$, then again by the maximum modulus theorem,
\[|y^*(f(y))|\le  |\frac{y^*}{\|y^*\|}( f(y))| =
\max_{|\lambda|= 1} |\frac{y^*}{\|y^*\|}( f(\lambda \frac{y}{\|y\|})
)|\le v(f).\] This completes the proof.
\end{proof}

\begin{prop}\label{prop:sameradius2}
Let $H$ be a subspace of $A_b(B_X: X)$ with a numerical boundary
$\Gamma$. Suppose that  a norm-one finite dimensional projection
$(\pi, F)$  is parallel to  $\Gamma$. Then for any $f\in H_F$,
\[ v_F(f) = v_X(f\circ \pi),\] where $v_X(f\circ \pi)$ is a numerical
radius as a function  $f\circ \pi: B_X\to X$.
\end{prop}
\begin{proof}
It is clear that $v_F(f) \le v_X(f\circ \pi)$. For the converse,
choose a sequence $\{(x_n, x^*_n)\}_{n=1}^\infty$ in $\Gamma$ such
that
\[ v_X(f\circ \pi)=\lim_n |x_n^* (f(\pi (x_n)))| = \lim_n \inner{x_n^*|_{F}, f(\pi(x_n))}.\] Since $\{\pi(x_n)\}_{n=1}^\infty$
is in the finite dimensional space $F$, we may assume that $\{\pi
(x_n)\}_{n=1}^\infty$ converges to $y\in B_F$ and
$\{x_n^*|_{F}\}_{n=1}^\infty$ converges to $y^*\in B_{F^*}$. Then
$|\inner{y^*, y}| = \|y^*\|\cdot \|y\|$. Thus by
Lemma~\ref{lem:basic},
\[ v_X(f\circ \pi) = |y^*(f(y))|\le v_F(f).\]
The proof is complete.
\end{proof}

Now we get the extensions of the results of E. Ed-dari \cite{Ed} and
the first named author \cite{K2} in the complex case.

\begin{thm}\label{thm:dari1} Let $H$ be a subspace of $A_b(B_X: X)$
with a numerical boundary $\Gamma$. Suppose that the Banach space
$X$ has $(FPA)$-property with $\{\pi_i, F_i\}_{i\in I}$ and that the
corresponding projections are parallel to $\Gamma$. Then
\[N(H)= \inf_{i\in I}  N(H_{F_i}).\] In fact, $N(H)$
is a decreasing limit of the right-hand side with respect to the
inclusion partial order.
\end{thm}
\begin{proof}
For any $f\in H_F$, $v_{F_i}(f)= v_X(f\circ \pi_i)$ by
Proposition~\ref{prop:sameradius2}. $v_{F_i}(f) = v_X(f\circ \pi)\ge
\|f\circ \pi\|N(H)= \|f\|N(H).$ Hence $N(H_{F_i}) \ge N(H)$ and it
is easy to see that if $F_i\subset F_j$, then $N(H)\le N(H_{F_j})\le
N(H_{F_i})$. Hence $N(H)\le \inf_{i\in I} N(H_{F_i})$. The converse
is clear by Proposition~\ref{prop:sameradius1}.
\end{proof}

\bigskip
The author {\em et al.} \cite{CGKM} introduced the concept of {\em
the polynomial numerical index of order $k$ of $E$} to be the
constant
\[n^{(k)}(E):=\inf \{v(P): P \in {\mathcal P}(^kE:E), \|P\|=1\}.\] The first named author \cite{K2} extended Ed-dari's
result(\cite{Ed}, Theorem 2.1) to the polynomial numerical indices
of (real or complex) $l_p$ of order $k$.

\begin{cor}
Let $k\ge 1$ and $1< p<\infty$. Then
\[ \lim_{m\to \infty} N(\mathcal{P}({\, }^k \ell_p^m)) = N(\mathcal{P}({\, }^k \ell_p)) \le
N(\mathcal{P}({\, }^k L_p(0,1)).\]
\[ \lim_{m\to \infty} N(A_b(B_{\ell_p^m}: \ell_p^m)) =
N(A_b(B_{\ell_p}: \ell_p))\le N(A_b(B_{L_p(0,1)}: L_p(0,1))).\]
\end{cor}
\begin{proof}
We give only the first part, since the proof of the next is similar.
Let $H=\mathcal{P}({\, }^k \ell_p)$. Then $\ell_p$ has the
$(FPA)$-property with  projections $\{\pi_i, F_i\}_{i=1}^\infty$,
where each $\pi_i$ is a $i$-th natural projections. Notice that
given projections are parallel to $\Pi(X)$. Hence $N(H) = \inf_{i\in
I} N(H_{F_i})$ by Theorem~\ref{thm:dari1}. Notice that $H_{F_i}$ is
isometrically isomorphic to $\mathcal{P}({\, }^k \ell_p^i)$.

On the other hand, if we let $H=\mathcal{P}({\, }^k L_p(0,1))$. Then
$L_p(0,1)$ has $(FPA)$-property with projections $\{\pi_i, F_i\}$,
where each $\pi_i$ is a conditional expectation with respect to a
sub-$\sigma$-algebra generated by finitely many disjoint subsets.
Hence $N(H) \ge \inf_{i\in I} N(H_{F_i})$. Notice also that $F_i$ is
isometrically isomorphic to $\ell_p^m$ for some $m$. So $H_{F_i}$ is
isometrically isomorphic to $\mathcal{P}({\, }^k \ell_p^m)$. The
proof is complete.
\end{proof}

Notice that if $X$ is locally uniformly convex, then every norm-one
projection is strong. Indeed, suppose that if $\pi:X\to F$ is a
norm-one projection and if $\{\pi(x_k)\}_{k=1}^{\infty}$ in $B_X$
converges to $y\in S_F$, then
\[ 1=\lim_k{\left\|\frac {\pi(x_k) + y}2 \right\|} =
\lim_k{\left\|\frac{\pi( x_k +y )}2\right\|}\le \lim_k{\left\|\frac{
x_k +y }2\right\|}\le 1\] shows that $\lim_k \|x_k + y\|=2$ and
$\lim_k \|x_k - y\|=0$ since $X$ is locally uniformly convex. For
its generalization to a strong complex extreme point, see
\cite[Proposition~3.1]{CHL}.

\begin{thm}
Suppose that the smooth Banach space $X$ has the $(FPA)$-property
with $\{\pi_i, F_i\}_{i\in I}$ and the corresponding projections are
strong and parallel to $\Pi(X)$. Then the set of all numerical and
norm strong peak functions in $A_{wu}(B_X:X)$ is dense.
\end{thm}
\begin{proof}
By Theorem~\ref{prop:3peak1}, the set of all polynomials $Q$ such
that there exists norm-one projection $\pi:=\pi_i:X\to F$ such that
$\pi \circ Q\circ \pi = Q$ and $Q|_F$ is a  norm and numerical peak
function as a mapping from $B_F$ to $F$ is dense in $A_{wu}(B_X:
X)$.

Fix corresponding $Q$ and $\pi$ and assume that
$v_F(Q)=|y^*_0(Q(y_0))|$ and $\|Q(y_1)\|=\|Q\|$ for some $(y_0^*,
y_0)\in \Pi(F)$ and $y_1\in B_F$, where $v_F(Q)$ is the numerical
radius of the map $Q|_F:B_F\to F$.

Suppose that there is a sequence $\{(x_k, x_k^*)\}_{k=1}^\infty$ in
$\Pi(X)$ such that  $\lim_k |x_k^*(Q(x_k))| = v(Q)$. Then
\[ |\inner{x^*_k, Q(x_k) }|
= |\inner{x^*_k|_F, Q(\pi(x_k))}|\to v(Q).\] We may assume that the
sequence $\{(\pi (x_k), x_k^*|_F)\}_{k=1}^{\infty}$ converges to
$(y, y^*)\in B_F\times B_{F^*}$  in the norm topology. So $v(Q) =
|y^*(Q(y))|\ge v_F(Q)$. Since $\pi$ is parallel to $\Pi(X)$,
$|\inner{y^*, y}|= \|y^*\|\cdot \|y\|$. By Lemma~\ref{lem:basic},
\[ v(Q)=|y^*(Q(y))|\le v_F(Q).\] So $v(Q)=|y^*(Q(y))|=v_F(Q)$.
Since $Q|_F$ is a numerical peak function, $\|y\|=1= \|y^*\|$ and
$y=y_0$ and $y^* = y_0^*$.

Since $\pi$ is strong,  $\lim_n x_n =y_0$. Let $x^*$ be the weak-$*$
limit point of the sequence $\{x_n^*\}$. Then $x^*(y) =1$ and
$\|x^*\|=1= \|x^*|_F\|$ and
\[ v(Q) = |x^*(Q(y))| = |y^*(Q(y))| = v_F(Q) \] implies that $x^*|_F=
y^*$ since $Q|_F$ is a numerical strong peak function. Hence $x^*$
is unique because $X$ is smooth. Therefore $\{x_n^*\}_{n=1}^\infty$
converges weak-$*$ to $x^*$. The proof is complete.
\end{proof}

\begin{thm}\label{thm:verystongnumerical}
Suppose that $X$ space has the $(FPA)$-property with $\{\pi_i,
F_i\}_{i\in I}$ and the corresponding projections are strong and
parallel to $\Pi(X)$. We also assume that each $\pi^*_i:X^* \to X^*$
is strong. Then the set of all very strong numerical and norm strong
peak functions is dense in $A_{wu}(B_X: X)$.
\end{thm}
\begin{proof}
By Theorem~\ref{prop:3peak1}, the set of all polynomials $Q$ such
that there exists norm-one projection $\pi:=\pi_i:X\to F$ such that
$\pi \circ Q\circ \pi = Q$ and $Q|_F$ is a norm and numerical peak
function as a mapping from $B_F$ to $F$ is dense in $A_{wu}(B_X:
X)$.

Fix corresponding $Q$ and $\pi$ and assume that
$v_F(Q)=|y^*_0(Q(y_0))|$ and $\|Q(y_1)\|=\|Q\|$ for some $(y_0^*,
y_0)\in \Pi(F)$ and $y_1\in B_F$, where $v_F(Q)$ is the numerical
radius of the map $Q|_F:B_F\to F$.

Suppose that there is a sequence $\{(x_k, x_k^*)\}_{k=1}^\infty$ in
$\Pi(X)$ such that  $\lim_k |x_k^*(Q(x_k))| = v(Q)$. Then
\[ |\inner{x^*_k, Q(x_k) }|
= |\inner{x^*_k|_F, Q(\pi(x_k))}|\to v(Q).\] We may assume that the
sequence $\{(\pi (x_k), x_k^*|_F)\}_{k=1}^{\infty}$ converges to
$(y, y^*)\in B_F\times B_{F^*}$  in the norm topology. So $v(Q) =
|y^*(Q(y))|\ge v_F(Q)$. Since $\pi$ is parallel to $\Pi(X)$ ,
$|\inner{y^*, y}|= \|y^*\|\cdot \|y\|$. By Lemma~\ref{lem:basic},
\[ v(Q)=|y^*(Q(y))|\le v_F(Q).\] So $v(Q)=|y^*(Q(y))|=v_F(Q)$.
Since $Q|_F$ is a numerical peak function, $\|y\|=1= \|y^*\|$ and
$y=y_0$ and $y^* = y_0^*$.

Since $\pi$ is strong,  $\lim_n x_n =y_0$. Fix $z^*\in S_{X^*}$ to
be a Hahn-Banach extension of $y^*$. Let $x^*$ be the weak-$*$ limit
point of the sequence $\{x_n^*\}_{n=1}^{\infty}$. Then $x^*(y) =1$
and $\|x^*\|=1= \|\pi^*(x^*)\|$ and
\[ v(Q) = |x^*(Q(y))| = |y^*(Q(y))| = v_F(Q) \] implies that $\pi^*(x^*)|_F=
y^*$ since $Q|_F$ is a numerical strong peak function so
$\pi^*(x^*)= \pi^*(x^*)$.

Hence $\lim_n \pi^* (x_n^*) = \pi^*(z^*)$ and $\|\pi^*(z^*)\|=1$.
Now we get $\|x^*_n-\pi^*(z^*)\|\to 0$ by the assumption. This shows
that $\lim_n \|x_n^*-\pi^*(z^*)\|=0$. Therefore $x^* = \pi^*(z^*)$
and $Q$ is a very strong numerical peak function at $(y,
\pi^*(z^*))$. This completes the proof.
\end{proof}

\begin{cor}\label{cor:littlep}
Suppose that  $X=\ell_p$ with $1<p<\infty$. Then the set of all very
strong numerical and norm strong peak functions is dense in
$A_{wu}(B_X: X)$.
\end{cor}
\begin{proof}
Let $\{\pi_i, F_i\}_{i=1}^\infty$ be a projections consisting of
$i$-th natural projections. Then these projections satisfy the
conditions in Theorem~\ref{thm:verystongnumerical}. The proof is
done.
\end{proof}

In fact, Corollary~\ref{cor:littlep} holds in general $L_p$ space if
$1<p<\infty$ as we see in Corollary~\ref{cor:largep}.

\begin{thm}\label{thm:smoothrnp}
Suppose that $X$ is a smooth Banach space with the Radon-Nikod\'ym
property. Then the set of all numerical strong peak functions is
dense in $A(B_X: X)$.
\end{thm}
\begin{proof}
An element $h\in A_b(B_X: X)$ is said to {\it strongly attain its
numerical radius} if there is $(x, x^*)\in \Pi(X)$ such that
whenever there is a sequence $\{(x_n, x_n^*)\}_{n=1}^{\infty}$ in
$\Pi(X)$ such that $\lim_n |x^*_n (h(x_n))| = v(h)$, there exist a
subsequence $\{(x_{n_k}, x_{n_k}^*)\}_{k=1}^{\infty}$ in $\Pi(X)$
and $\lambda\in S_\mathbb{C}$ such that $\{(x_{n_k},
x_{n_k}^*)\}_{k=1}^{\infty}$ converges to $(\lambda x,
\overline{\lambda} x^*)$ in $\Pi(X)$.

claim: The set of all elements which strongly attain their numerical
radius is dense in $A(B_X: X)$

Fix $f\in A_b(B_X:X)$ and $\epsilon>0$. Define for each $x\in B_X$,
\[ \varphi(x): = \max\{ |x^*(f(\lambda x))| : \lambda\in \mathbb{C}, |\lambda|\le 1, x^*(x) = \|x\|, x^*\in S_X\}.\]
We claim that $\varphi$ is upper semi-continuous. Indeed, if the
sequence $\{x_n\}_{n=1}^{\infty}$ converges to $x$, then for each
$n\ge 1$, choose $\lambda_n$ such that $\varphi(x_n) =
|x_n^*(f(\lambda_n x_n))|$ and let $x^*$ be the weak-$*$ limit point
of $\{x_n^*\}$. Then since $x_n^*(x_n)=\|x_n\|$, we get $x^*(x) =
\|x\|$. We may assume that the sequence $\{x_n\}_{n=1}^{\infty}$ and
$\{\lambda_n\}_{n=1}^{\infty}$ converge to $x^*$ and $\lambda$,
respectively. Then
\[\lim_{n\to \infty} \varphi(x_n) = \lim_n |x_n^* (f(\lambda_n x_n))|= |x^*(f(\lambda x))| \le
\varphi(x).\] Hence it is easy to see that $\limsup_n \varphi(x_n)
\le \varphi(x)$.

By the perturbed optimization theorem of Bourgain and Stegall
(\cite{Bo}, \cite{S}), there is $y^*$ such that $\|y^*\|<\epsilon$
and $\varphi+ {\rm Re} y^*$ strongly exposes $B_X$ at $x_0$. Then
$y^*(x_0)\neq 0$. Otherwise,
\begin{eqnarray*} \varphi(x_0)
&= \sup\{ \varphi(x) + {\rm Re} y^*(x) : x\in B_X\}\\
&= \sup\{ \varphi(x) +  |y^*(x)| : x\in B_X\}
\end{eqnarray*} and $\varphi(x_0)+{\rm Re} y^*(x_0) =\varphi(-x_0) + {\rm Re} y^*(-x_0)$.
Since $\varphi + {\rm Re} y^*$ strongly exposes $B_X$ at $x_0$, we
get $x_0=0$. It is clear that $\varphi(0)=\sup_{x\in B_X} \varphi(x)
\ge v(f)$. This implies that $\varphi(0) = \|f(0)\|= v(f)$ by
Lemma~\ref{lem:basic}. Choose a sequence $\{(x_n,
x_n^*)\}_{n=1}^{\infty}$ in $\Pi(X)$ such that \[\lim_n
|x_n^*(f(x_n))|=v(f).\] Then $|x_n^* (f(x_n))| \le \varphi(x_n)
+|y^*(x_n)| = \varphi(\lambda_n x_n) + {\rm Re} y^*(\lambda_n x_n)
\le \varphi(0)=v(f)$ for a suitable sequence $\{\lambda_n\}$ in
$S_\mathbb{C}$. So $\lim_n \varphi(\lambda_n x_n) + {\rm Re}
y^*(\lambda_n x_n) =\varphi(0)$. Since $\varphi + {\rm Re} y^*$
strongly exposes $B_X$ at 0, $\{\lambda_n x_n\}_{n=1}^{\infty}$
converges to $0$. This is a contradiction to $\lim_n \|\lambda_n x_n
\|=1$.

Now we get $\|x_0\|=1$. Indeed, it is clear that $x_0\neq 0$. If
$0<\|x_0\|<1$, then
\begin{eqnarray*} \varphi(x_0) + {\rm Re}
y^*(x_0) &= \sup\{ \varphi(x) + {\rm Re} y^*(x) : x\in B_X\}\\
&= \sup\{ \varphi(x) +  |y^*(x)| : x\in B_X\}
\end{eqnarray*}
shows that ${\rm Re} y^*(x_0) = |y^*(x_0)|$ and
\[ \varphi(x_0) + |y^*(x_0)| < \varphi(\frac{x_0}{\|x_0\|}) +
|y^*(\frac{x_0}{\|x_0\|})| = \varphi(\frac{x_0}{\|x_0\|}) + {\rm Re}
y^*(\frac{x_0}{\|x_0\|}).\] This is a contradiction to the fact that
$\varphi +{\rm Re} y^*$ strongly exposes $B_X$ at $x_0$.

Fix $N\ge 1$. There exist  $\lambda_0\in S_\mathbb{C}$, $x_0^*\in
S_{X^*}$, and $x_0^*(x_0)=1$ such that $\varphi(x_0)=
|x_0^*(f(\lambda_0 x_0))|$. Define $h:B_X\to X$ by
\[ h(x):= f(x) + \lambda_1 (\overline{\lambda_0} x_0^*(x))^{N-1}
y^*(x)x_0,\] where the complex number $\lambda_1\in S_\mathbb{C}$ is
properly chosen so that
\[ |x^*_0(f(\lambda_0 x_0)) +\lambda_1 \lambda_0 y^*(x_0)| = |x^*_0(f(\lambda_0 x_0))|+|y^*(x_0)|.\]

It is clear that $h\in A(B_X:X)$ and notice that we get for every
$(x, x^*)\in \Pi(X)$, \begin{eqnarray} |x^*(h(x))| &\le& |x^*(f(x))|
+ |y^*(x)| \le \varphi(x) + |y^*(x)| \label{eq1}\\&\le& \sup\{
\varphi(x) + |y^*(x)| : x\in B_X\} \nonumber\\&=& \sup\{ \varphi(x)
+ {\rm Re} y^*(x) : x\in B_X\}=\varphi(x_0) + {\rm Re}
y^*(x_0).\nonumber\end{eqnarray}

Note that $(\lambda_0x_0, \overline{\lambda_0}x_0^*)\in \Pi(X)$.
Hence $v(h) = \varphi(x_0) + {\rm Re} y^*(x_0)$ because ${\rm Re}
y^*(x_0) = |y^*(x_0)|$ and
\begin{eqnarray*}v(h)&\geq& | \overline{\lambda_0}x_0^*(h(\lambda_0 x_0))|=|x_0^*(h(\lambda_0 x_0))|
= |x_0^*(f(\lambda_0x_0))+\lambda_1\lambda_0 y^*(x_0)|\\&=&
|x^*_0(f(\lambda_0 x_0))|+ |y^*(x_0)|= \varphi(x_0) + |y^*(x_0)| =
\varphi(x_0) + {\rm Re} y^*(x_0).\end{eqnarray*}

We shall show that $h$ strongly attains its numerical radius at
$(x_0, x_0^*)$. Suppose that $\lim_n |x_n^*(h(x_n))|=v(h) =
\varphi(x_0) + {\rm Re} y^*(x_0)$. Choose a sequence  $\{\alpha_n\}$
of complex numbers so that $|\alpha_n|=1$ and
\[ \varphi(x_n) + |y^*(x_n)| = \varphi(\alpha_n x_n) + {\rm Re} y^*(\alpha_n x_n),\ \ \ \forall n\ge 1.\]
Then (\ref{eq1}) shows that $\lim_{n\to \infty}\varphi(\alpha_n x_n)
+ {\rm Re} y^*(\alpha_n x_n) = \varphi(x_0) + {\rm Re} y^*(x_0).$
Since $\varphi+ {\rm Re} y^*$ strongly exposes $B_X$ at $x_0$,
$\{\alpha_n x_n\}$ converges to $x_0$. Hence there is a subsequence
of $\{x_n\}$ which converges to $\alpha x_0$ for some $|\alpha|=1$.
For any weak-$*$ limit point $x^*$ of $\{x_n^*\}$, $x^*(\alpha x_0)
= 1$. Since $X$ is smooth, $x^* = \overline{\alpha} x_0^*$. This
shows that the subsequence $\{x_n^*\}$ converges weak-$*$ to
$\overline\alpha x^*_0$. This shows that $h$ strongly attains its
numerical radius at $(x_0, x_0^*)$. Notice that $\|h - f\|\le
\epsilon$. This proves the claim.

Now choose $h\in A(B_X: X)$ which strongly attains its numerical
radius at $(x_0, x_0^*), \|h-f\|<\frac{\epsilon}{2}$, and $v(h) =
|x_0^*(h(\lambda_0x_0))|$ for some $(x_0, x_0^*)\in \Pi(X)$ and
$\lambda_0\in S_\mathbb{C}$. Choose a peak function $g\in
A(B_{\mathbb{C}})$ at $\lambda_0$ with $\|g\|<\frac{\epsilon}{2}$.
Define
\[ u(x):= h(x) + \eta g(x_0^*(x))x_0,\] where $\eta\in S_\mathbb{C}$
is chosen to be $ |x_0^*(h(\lambda_0 x_0)) + \eta g(\lambda_0)| =
|x_0^*(h(\lambda_0 x_0))| + |g(\lambda_0)|.$ Then \begin{eqnarray*}
|x_0^*(u(\lambda_0 x_0))| &=& |x_0^*(h(\lambda_0 x_0)) + \eta
g(\lambda_0)| \\&=&  |x_0^*(h(\lambda_0 x_0))| + |g(\lambda_0)| =
v(h) + \|g\|.\end{eqnarray*} For each  $(x, x^*)\in \Pi(X)$, we have
\[ |x^*(u(x))| \le |x^*(h(x))| + |g(x^*_0(x))| \le v(h) + \|g\|.\]
So $v(u) = v(h) + \|g\|$. Now we claim that $u$ is a numerical
strong peak function at $(\lambda_0 x_0, \overline{\lambda_0}
x_0^*)$. If there is a sequence $\{(x_n, x_n^*)\}$ in $\Pi(X)$ with
$\lim_n |x_n^*(u(x_n))| = v(u)$, then
\[ |x_n^*(u(x_n))| \le |x_n^*(h(x_n))| + |g(x_0^*(x_n))| \le v(h) +
\|g\|.\] Hence $\lim_n |x^*_N( h(x_n))| = v(h)$ and $\lim_n
|g(x_0^*(x_n))|= \|g\|$. Since $g$ is a strong peak function at
$\lambda_0$, we get $\lim_n x_0^*(x_n)=\lambda_0$.

For any subsequence of $\{(x_n, x_n^*)\}_{n=1}^\infty$, there are a
further subsequence $\{(y_n, y_n^*)\}$ and $\eta \in S_\mathbb{C}$
such that $\lim_n y_n = \eta x_0$ and $w^*-\lim_n y_n^* =
\overline\eta x_0^*$. Since $\lim_n x_0^*(y_n) = \eta$, $\eta
=\lambda_0$. This implies that $\lim_n x_n = \lambda_0x_0$ Let $x^*$
be the weak-$*$ limit point of $\{x_n^*\}_{n=1}^\infty$. Since $X$
is smooth, $x^*(x_0)=\overline{\lambda_0}$ implies that
$x^*=\overline{\lambda_0}x_0^*$. Therefore the weak-$*$ limit of
$\{x_n^*\}_{n=1}^\infty$ is $\overline{\lambda_0}x_0^*$. Thus $u$ is
a numerical strong peak function at $(\lambda_0 x_0,
\overline{\lambda_0} x_0^*)$. Notice also that $\|f-u\|\le
\epsilon$. This completes the proof.
\end{proof}

\begin{rem}
In the proof of Theorem~\ref{thm:smoothrnp}, it is shown that the
set of all strongly numerical radius attaining elements in
$\mathcal{P}({}^N X:X)$ is dense if $X$ is a smooth Banach space
with the Radon-Nikod\'ym property.
\end{rem}

{\bf Question.} Suppose that $X$ is a smooth Banach space with the
Radon-Nikod\'ym property. Is it true that the set of all elements
which are norm and numerical strong peak functions is dense in
$A(B_X: X)$?

\begin{prop}\label{prop:Gdelta}Let $X, Y$ be Banach spaces.
Let $A$ be a subspace of $C_b(B_X: Y)$. Then the set of all strong
peak functions in $A$ is a $G_\delta$-subset of $A$. In case that
$\Pi(X)$ is a complete metrizable space and $X=Y$, then the set of
all numerical strong peak functions in $A$ is a $G_\delta$-subset of
$A$. In particular, if $X^*$ is locally uniformly convex, then
$\Pi(X)$ is  complete metrizable and the net convergence of
$\{(x_\alpha, x_\alpha^*)\}_\alpha$ in $\tau$-topology implies the
convergence of each component in norm.
\end{prop}
\begin{proof}Notice that $f\in A$ is a strong peak
function at $x_0$ if and only if for each neighborhood $V$ of $x_0$,
there is $\epsilon>0$ such that $\{x\in B_X: \|f(x)\|\ge
\|f\|-\epsilon\}$ is contained in $V$.

For each $f\in A$ and $n\ge 1$, let
\[\Delta(f,n)=  \left\{x\in B_X: \|f(x)\|\ge \|f\|-\frac 1n \right\}.\]
For each natural number $N\ge 1$, define
\[ S_N: = \bigcup_{n=1}^\infty \left\{ f\in A\setminus \{0\} :{\rm diam} \Delta(f,n) \le \frac
1N\right\},\] where ${\rm diam} (C)= \inf \{d(x,y):x,y\in A\}$ for a
metric space $(C, d)$. Notice that the set of all  strong peak
functions is $S:=\cap_{N=1}^\infty S_N$. So we have only to show
that each $S_N$ is an open subset of $A$. Fix $f\in S_N$, then there
is $n\ge 1$ such that $n\|f\|>1$ and $\Delta(f,n)\le \frac 1N.$  If
$\|g- f\|\le 1/(3n)$, then $g\neq 0$ and $\Delta(g, 3n) \subset
\Delta(f, n)$. This shows that $f+ \frac{1}{3n} B_A$ is a subset of
$S_N$. Hence $S_N$ is an open subset of $A$.

For the second case, notice that $f\in A$ is a numerical strong peak
function at $(x_0, x^*_0)$ if and only if for each
$\tau$-neighborhood $V$ of $(x_0, x^*_0)$, there is $\epsilon>0$
such that $\{ (x, x^*)\in \Pi(X): |x^*(f(x))|\ge v(f) - \epsilon\}$
is contained in $V$. Suppose that $\Pi(X)$ is a complete metric
space with metric $d$. Similarly for each $f\in A$ and $n\ge 1$,
define \[\tilde\Delta(f, n) = \left\{(x,x^*)\in \Pi(X)
:|x^*(f(x))|\ge v(f)-\frac 1n \right\}.\]

For each natural number $N\ge 1$, define
\[ S_N := \bigcup_{n=1}^\infty \left\{ f\in A\setminus \{0\} : {\rm diam}\tilde\Delta(f, n)
 \le \frac 1N\right\}.\] Notice that the set of all
numerical strong peak functions is $S:=\cap_{N=1}^\infty S_N$. So we
have only to show that each $S_N$ is an open subset of $A$. Fix
$f\in S_N$, then there is $n\ge 1$ such that $n\|f\|>1$ and
$\tilde\Delta(f, n) \le \frac 1N.$ If  $\|g- f\|\le 1/(3n)$, then
$g\neq 0$ and $\tilde\Delta(g, 3n)\subset \tilde\Delta(f, n)$. This
shows that $f+ \frac{1}{3n} B_A$ is a subset of $S_N$. Hence $S_N$
is an open subset of $A$.

Finally, suppose that $X^*$ is locally uniformly convex. Then define
a function $d$ in $\Pi(X)\times \Pi(X)$ to be
\[ d((x, x^*), (y,y^*)) := \|x-y\| + \|x^*- y^*\|.\]
It is clear that $d$ is a complete metric in $\Pi(X)$ and it is also
clear that the $d$-convergence implies the $\tau$-convergence. For
the converse, if the net $\{(x_\alpha, x^*_\alpha)\}_\alpha$
converges to $(z, z^*)$ in $\tau$-topology, then $\lim_\alpha
\|z-x_\alpha\|=0$ and $\{x_\alpha^*\}_\alpha$ converges weak-$*$ to
$z^*$ with $z^*(z)=1$. So
\[ 1\le \liminf_\alpha \left\|\frac{x_\alpha^* + z^*}2 \right\| \le
\limsup_\alpha \left\|\frac{x_\alpha^* + z^*}2 \right\|\le 1.\]
Since $X^*$ is locally uniformly convex,  $\lim_\alpha \|x_\alpha^*
+ z\|=2$ implies that $\lim_\alpha \|z^* - x_\alpha^*\|=0$. Hence
the net $\{(x_\alpha, x^*_\alpha)\}$ converges to $(z, z^*)$ in the
$d$-metric topology. This completes the proof.
\end{proof}

It is shown in \cite{CLS} that the set of strong peak functions in
$A(B_X:X)$ is dense if $X$ has the Radon-Nikod\'ym property. Hence
by Theorem~\ref{thm:smoothrnp} and Proposition~\ref{prop:Gdelta} we
get the following.

\begin{cor}
Let $X$ be a complex Banach space with the Radon-Nikod\'ym property
and $X^*$ is locally uniformly convex. Then the set of all norm and
numerical strong peak functions in $A(B_X: X)$ is a dense
$G_\delta$-subset of $A(B_X: X)$. In particular, every numerical
strong peak function is a very strong numerical peak function.
\end{cor}

\begin{cor}\label{cor:largep}
Let $1<p<\infty$ and $X= L_p(\mu)$ for some measure space $(\Omega,
\Sigma, \mu)$. Then the set of all norm and numerical strong  peak
functions in $A(B_X: X)$ is a dense $G_\delta$-subset of $A(B_X:X)$.
In particular, every numerical strong peak function is a very strong
numerical peak function.
\end{cor}

In case that $X$ is finite dimensional, it is clear that $\Pi(X)$ is
a compact metric space. Hence by Theorem~\ref{prop:peak}, we get the
following.

\begin{cor}\label{cor:finitecor}
Let $X$ be a finite dimensional Banach space. Then the set of all
norm and numerical strong peak functions in $A(B_X: X)$ is a dense
$G_\delta$-subset of $A(B_X: X)$.
\end{cor}

\begin{prop}\label{prop:Gdelta2}
Let $X$ be a separable Banach space. Then $\Pi(X)$ is a complete
metrizable.
\end{prop}
\begin{proof}
Let $\{x_n\}_{n=1}^\infty$ be a dense subset in $B_X$. Then in
$B_{X^*}$, the metric
\[ d(x^*, y^*): = \sum_{n=1}^\infty \frac{|x^*(x_n)-
y^*(x_n)|}{2^n}\] induces the same topology as the weak-$*$ topology
in $B_{X^*}$. Define a function $d_1: \Pi(X)\times \Pi(X) \to
[0,\infty)$ to be
\[ d_1((x, x^*), (y, y^*)) := \|x-y\| + d(x^*, y^*).\]
It is clear that $d_1$ induces the $\tau$-topology in $\Pi(X)$. So
we have only to show that $d_1$ is a complete metric. Suppose that
$\{(x_n, x_n^*)\}$ is a $d_1$-Cauchy sequence. Then it is clear that
there is $x\in S_X$ such that $\lim_n \|x_n - x\|=0$. Notice that
$\lim_n x^*_N(x_k)$ exists for each $k\ge  1$. Let $x^*$ be the
weak-$*$ limit point of $\{x_n^*\}_{n=1}^\infty$. Then $x^*(x_k) =
\lim_n x^*_N(x_k)$ for each $k\ge 1$. Hence $\{x_n^*\}_n$ converges
weak-$*$ to $x^*$. This completes the proof.
\end{proof}

\begin{thm} Let $X=\ell_1$. Then the set of all norm and
numerical strong peak functions in $A(B_X: X)$ is a dense
$G_\delta$-subset of $A(B_X:X)$.
\end{thm}
\begin{proof}Since $\ell_1$ has the Radon-Nikod\'ym property, the set of strong peak functions in
$A(B_X:X)$ is dense \cite{CLS}.  Hence by
Proposition~\ref{prop:Gdelta}, the set of all strong peak function
in $A(B_X:X)$ is a dense $G_\delta$-subset of $A(B_X:X)$. By
Proposition~\ref{prop:Gdelta} and~\ref{prop:Gdelta2}, the set of all
numerical strong peak functions is a $G_\delta$-subset of
$A(B_X,X)$. So we have only to show that the set of all numerical
strong peak functions in $A(B_X:X)$ is dense.

Fix $f\in A_b(B_X:X)$ and $\epsilon>0$. Define for each $x\in B_X$,
\[\varphi(x): = \max\{ |x^*(f(\lambda x))| : \lambda\in \mathbb{C}, |\lambda|\le 1, x^*(x) = \|x\|, x^*\in S_X\}.\]
In the proof of Theorem~\ref{thm:smoothrnp}, we showed that
$\varphi$ is upper semi-continuous. Since $\ell_1$ has the
Radon-Nikod\'ym property, there is $y^*$ such that
$\|y^*\|<\epsilon$ and $\varphi+ {\rm Re} y^*$ strongly exposes
$B_X$ at $x_0$. Then $y^*(x_0)\neq 0$ and $\|x_0\|=1$ as the proof
of Theorem~\ref{thm:smoothrnp}. Notice that \begin{eqnarray*}
\varphi(x_0) + {\rm Re}
y^*(x_0) &=& \sup\{ \varphi(x) + {\rm Re} y^*(x) : x\in B_X\}\\
&=& \sup\{ \varphi(x) +  |y^*(x)| : x\in B_X\}\\
&=& \varphi(x_0)+ |y^*(x_0)|.
\end{eqnarray*}

Then there exist $\lambda_0\in S_\mathbb{C}$ and $x_0^*\in
S_{\ell_\infty}$ such that $\varphi(x_0) = |x_0^*
(f(\lambda_0x_0))|$ and $x_0^*(x_0)=1=\|x_0\|$. So it is easy to see
that if $i\in {\rm supp}(x_0)$, then $x_0^*(i)= \overline{{\rm
sign}(x_0(i))}$.

For each $i\ge 1$, let $x_1^*$ as
\[ x_1^*(i): = \left\{
              \begin{array}{ll}
                \overline{{\rm sign}(x_0(i))}, & \hbox{if } x_0(i)\neq 0 \\
                0, & \hbox{otherwise.}
              \end{array}
            \right. \]
Notice that
\begin{eqnarray*} \varphi(x_0) &= \sup_{ y^*\in B_{\ell_\infty}, {\rm supp}(y^*)\cap {\rm
supp}(x_1^*)=\emptyset}\{ |x^*_1 (f(\lambda_0 x_0)) +
y^*(f(\lambda_0 x_0))|\}
\\
&= \sup_{y^*\in B_{\ell_\infty}, {\rm supp}(y^*)\cap {\rm
supp}(x_1^*)=\emptyset}\{ |x^*_1 (f(\lambda_0 x_0))| +
|y^*f((\lambda_0 x_0))|\}\\
&= |x_1^*(f(\lambda_0 x_0))| + \sum_{i\not\in {\rm supp}(x_0)}
|\inner{e_i^*, f(\lambda_0 x_0)}|.
\end{eqnarray*} Now let $x_2^*$ as
\[  x_2^*(i): = \left\{
              \begin{array}{ll}
                \overline{{\rm sign}(f(\lambda_0 x_0))(i)}, & \hbox{  if  } x_0(i)= 0, i\in {\rm supp}(f(\lambda_0 x_0)) \\
                0, & \hbox{  otherwise.}
              \end{array}
            \right. \]
Then $\varphi(x_0) = |x_1^*(f(\lambda_0x_0))| +
|x_2^*(f(\lambda_0x_0))|$.

In the first case, suppose that $\mathbb{N}\neq {\rm supp}(x_0)~
\cup ~{\rm supp}(f(\lambda_0 x_0))$ and  \[\sum_{i\not\in {\rm
supp}(x_0)} |\inner{e_i^*, f(\lambda_0 x_0)}|\neq 0.\]

Choose a peak function $g\in A_u(B_\mathbb{C})$ at $\lambda_0$ with
$0<\|g\|\le \epsilon$ and $0\neq y\in \ell_1$ such that ${\rm
supp}(y)$ is the complement of ${\rm supp}(x_0)\cup {\rm
supp}(f(\lambda_0 x_0))$ and $\|y\|\le \epsilon$. Then there is a
unique element $x_3^*\in S_{\ell_\infty}$ such that ${\rm
supp}(x_3^*)={\rm supp}(y)$ and $x_3^*(y)=\|y\|$.

Define a function $h\in A(B_X: X)$ as
\[ h(x): = f(x) + \eta_1 y^*(x) x_0 + \eta_2 g(x_1^*(x))x_0 +y,~ \forall x\in B_X\] and
define $z^*: = x_1^* + \xi_2 x_2^* + \xi_3 x_3^*\in S_{l_{\infty}}$,
where if $x_1^*(f(\lambda_0 x_0))\neq 0$, then $\eta_1$, $\eta_2$,
$\xi_2$ and $\xi_3$ are uniquely determined complex numbers in
$S_\mathbb{C}$ such that
\begin{eqnarray}\label{eq:1sameargument}|z^*(h(\lambda_0 x_0))| &=& |x_1^*(f(\lambda_0 x_0))+ \xi_2 x_2^*(f(\lambda_0
x_0))+ \eta_1 y^*(\lambda_0 x_0) + \eta_2\|g\|  + \xi_3\|y\||~\nonumber\\
&=& |x_1^*f(\lambda_0 x_0)|+ |x_2^*f(\lambda_0 x_0)|+
|y^*(\lambda_0x_0)| + \|g\| +
 + \|y\|\\
&=& \varphi(x_0) + |y^*(x_0)| + \|g\| +
\|y\|,\nonumber\end{eqnarray} if $x_1^*(f(\lambda_0 x_0))=0$, then
just take $\eta_1=1$ and choose $\eta_2$, $\xi_2$ and $\xi_3$ as
uniquely determined complex numbers in $S_\mathbb{C}$ satisfying
(\ref{eq:1sameargument}).

Notice that $(\lambda_0 x_0, \overline\lambda_0 z^*)\in \Pi(X)$.
Hence $v(h) \ge \varphi(x_0) + |y^*(x_0)| + \|g\| + \|y\|.$ For any
$(x, x^*)\in \Pi(X)$,
\[ |x^*(h(x))| \le \varphi(x) + |y^*(x)| + \|g\| + \|y\|\le \varphi(x_0) +
|y^*(x_0)|+\|g\| + \|y\|.\] Hence $v(h) =\varphi(x_0) + |y^*(x_0)| +
\|g\| + \|y\|.$

We claim that $h$ is a numerical strong peak function at $(\lambda_0
x_0, \overline{\lambda_0} z^*)$. Indeed, if there is a sequence
$(x_n, x_n^*)\in \Pi(X)$ such that $\lim_n |x_n^*(h(x_n))|=v(h)$,
then there is a sequence $\{\tau_n\}$ in $S_\mathbb{C}$ such that
\begin{eqnarray*}
|x_n^*(h(x_n))| &\le& |x_n^*(f(x_n))| + |y^*(x_n)| + |x_n^*(y)| +
|g(x_1^*(x_n))|\\
&\le& \varphi(\tau_n x_n) + {\rm Re} y^*(\tau_n x_n) + |x_n^*(y)| +
|g(x_1^*(x_n))|\\
&\le& \varphi( x_0) + |y^*(x_0)| + \|g\| + \|y\| =v(h).
\end{eqnarray*}
Therefore $\lim_n \tau_n x_n = x_0$ and $\lim_n
x_1^*(x_n)=\lambda_0$. So it is easy to see that $\lim_n \tau_n =
\overline{\lambda_0 }$. This implies that $\lim_n x_n = \lambda_0
x_0$.

Let $x^*$ be the weak-$*$ limit point of $\{x_n^*\}_{n=1}^\infty$.
Then \begin{eqnarray*} v(h) &=& |x^*(h(\lambda_0x_0))|\\ &=&|
x^*(f(\lambda_0x_0)) + \eta_1
y^*(\lambda_0x_0) x^*(x_0) + \eta_2 \|g\|x^*(x_0) +x^*(y)|\\
& \le& | x^*(f(\lambda_0 x_0))| + |y^*(x_0)|~ |x^*(x_0)| +  \|g\|~|x^*(x_0)| +|x^*(y)|\\
& \le& \varphi(x_0) + |y^*(x_0)| + \|g\| + \|y\| = v(h)
\end{eqnarray*} shows that $|x^*(x_0)|=1$ and $|x^*(y)|= \|y\|$. So
$x^* = \xi_1' x_1^* + \xi_3' x_3^* + x_4^*$ for some $x_4^*\in
B_{\ell_\infty}$ with ${\rm supp}(x_4^*)=[{\rm supp}(x_0)~\cup~ {\rm
supp}(y)]^c$ and $\xi_1', \xi_3'\in S_\mathbb{C}$. So
\begin{eqnarray} v(h) &=& |x^*(h(\lambda_0))|\\ &=&|\xi_1'
x^*_1(f(\lambda_0x_0)) + x_4^*(f(\lambda_0 x_0)) +\xi_1' \eta_1
y^*(\lambda_0x_0) + \eta_2 \|g\|\xi_1' +\xi_3'\|y\||\nonumber\\
&=&| x^*_1(f(\lambda_0x_0)) + \overline{\xi_1'} x_4^*(f(\lambda_0
x_0)) + \eta_1 y^*(\lambda_0x_0) + \eta_2 \|g\|
+\overline{\xi_1'}\xi_3'\|y\||\label{eq2:same3}\\
& \le& | x^*_1(f(\lambda_0x_0))| +  |x_4^*(f(\lambda_0 x_0))| + |
y^*(x_0)| + \|g\| +\|y\|\nonumber\\
&\le& | x^*_1(f(\lambda_0x_0))| +  |x_2^*(f(\lambda_0 x_0))| + |
y^*(x_0)| + \|g\| +\|y\| = v(h)\nonumber
\end{eqnarray} shows that $|x_4^*(f(\lambda_0 x_0))|= |x_2^*(f(\lambda_0
x_0))|$. Notice that ${\rm supp}(x_4^*)={\rm supp}(x_2^*)$ and
\[|x_2^*(f(\lambda_0 x_0))|= \sum_{i\in supp(x_2^*)} |\inner{e_i^*,
f(\lambda_0 x_0)}|.\] So $x_4^* = \xi_2' x_2^*$ for some $\xi_2'\in
S_\mathbb{C}$. By (\ref{eq:1sameargument}) and (\ref{eq2:same3}),
\begin{eqnarray*}
v(h) &=& | x^*_1(f(\lambda_0x_0)) + \overline{\xi_1'}\xi_2'
x_2^*(f(\lambda_0 x_0)) + \eta_1 y^*(\lambda_0x_0) + \eta_2 \|g\|
+\overline{\xi_1'}\xi_3'\|y\||\\ &=&|x_1^*(f(\lambda_0 x_0))+ \eta_1
y^*(\lambda_0 x_0) + \eta_2\|g\| + \xi_2 x_2^*f(\lambda_0 x_0) +
\xi_3\|y\||.\end{eqnarray*} Since $x^*(x_0)=\overline{\lambda_0}$,
we get $\xi_1' = \overline{\lambda_0}$, $\xi_2'
=\overline{\lambda_0}\xi_2$ and $\xi_3'=\overline{\lambda_0}\xi_3$.
Hence $x^* = \overline{\lambda_0}z^*$ and $\{x_n^*\}$ converges
weak-$*$ to $\overline{\lambda_0} z^*$. Thus $h$ is a numerical
strong peak function at $(\lambda_0 x_0, \overline{\lambda_0} z^*)$.

In the second case, we assume that $\sum_{i\not\in {\rm supp}(x_0)}
|\inner{e_i^*, f(\lambda_0 x_0)}|= 0$ and $\mathbb{N}\neq {\rm
supp}(x_0)$.

Then ${\rm supp}(f(\lambda_0 x_0))\subset {\rm supp}(x_0)$. Choose a
peak function $g\in A_u(B_\mathbb{C})$ at $\lambda_0$ with $\|g\|\le
\epsilon$ and $y\in \ell_1$ such that ${\rm supp}(y)$ is the
complement of ${\rm supp}(x_0)$ and $\|y\|\le \epsilon$. Then there
is a unique element $x_3^*\in S_{\ell_\infty}$ such that ${\rm
supp}(x_3^*)={\rm supp}(y)$ and $x_3^*(y)=\|y\|$.

Define a function $h\in A(B_X:X)$ as
\[ h(x): = f(x) + \eta_1 y^*(x) x_0 + \eta_2 g(x_1^*(x))x_0 +y,~\forall x\in B_X\] and
define $z^*: = x_1^* + \xi_3 x_3^*$, where if $x_1^*(f(\lambda_0
x_0))\neq 0$, then $\eta_1$, $\eta_2$ and $\xi_3$ are uniquely
determined complex numbers in $S_\mathbb{C}$ such that
\begin{eqnarray}\label{eq:1sameargument2}|z^*(h(\lambda_0 x_0))| &=& |x_1^*(f(\lambda_0 x_0))+
\eta_1 y^*(\lambda_0 x_0) + \eta_2\|g\| + \xi_3\|y\||\nonumber\\
&=& |x_1^*(f(\lambda_0 x_0))|+  |y^*(\lambda_0x_0)| + \|g\| +
 \|y\|\\
&=& \varphi(x_0) + |y^*(x_0)| + \|g\| +
\|y\|,\nonumber\end{eqnarray} if $x_1^*(f(\lambda_0 x_0))=0$, then
take $\eta_1=1$ and choose $\eta_2$ and $\xi_3$ in $S_\mathbb{C}$ as
uniquely defined numbers satisfying (\ref{eq:1sameargument2}).

Notice that $(\lambda_0 x_0, \overline\lambda_0 z^*)\in \Pi(X)$.
Hence $v(h) \ge \varphi(x_0) + |y^*(x_0)| + \|g\| + \|y\|.$ For any
$(x, x^*)\in \Pi(X)$,
\[ |x^*(h(x))| \le \varphi(x) + |y^*(x)| + \|g\| + \|y\|\le \varphi(x_0) +
|y^*(x_0)|+\|g\| + \|y\|.\] Hence $v(h) =\varphi(x_0) + |y^*(x_0)| +
\|g\| + \|y\|.$

We claim that $h$ is a numerical strong peak function at $(\lambda_0
x_0, \overline{\lambda_0} z^*)$. Indeed, if there is a sequence
$(x_n, x_n^*)\in \Pi(X)$ such that $\lim_n |x_n^*(h(x_n))|=v(h)$,
then there is a sequence $\{\tau_n\}$ in $S_\mathbb{C}$ such that
\begin{eqnarray*}
|x_n^*(h(x_n))| &\le& |x_n^*(f(x_n))| + |y^*(x_n)| + |x_n^*(y)| +
|g(x_1^*(x_n))|\\
&\le& \varphi(\tau_n x_n) + {\rm Re} y^*(\tau_n x_n) + |x_n^*(y)| +
|g(x_1^*(x_n))|\\
&\le& \varphi( x_0) + |y^*(x_0)| + \|g\| + \|y\| =v(h).
\end{eqnarray*}
Therefore $\lim_n \tau_n x_n = x_0$ and $\lim_n
x_1^*(x_n)=\lambda_0$. So it is easy to see that $\lim_n \tau_n =
\overline{\lambda_0 }$. This implies that $\lim_n x_n = \lambda_0
x_0$.

Let $x^*$ be the weak-$*$ limit point of $\{x_n^*\}_{n=1}^\infty$.
Then \begin{eqnarray} v(h) &=& |x^*(h(\lambda_0))|\\
&=&|x^*(f(\lambda_0x_0)) + \eta_1 y^*(\lambda_0x_0) x^*(x_0 )+ \eta_2 \|g\|x^*(x_0) +x^*(y)|\nonumber\\
&\le& | x^*(f(\lambda_0 x_0))| + |y^*(x_0)| ~|x^*(x_0)| +  \|g\|~|x^*(x_0)| +|x^*(y)|\label{eq2:same4}\\
&\le& \varphi(x_0) + |y^*(x_0)| + \|g\| + \|y\| = v(h)\nonumber
\end{eqnarray} shows that $|x^*(x_0)|=1$ and $|x^*(y)|=\|y\|$. So $x^* = \xi_1' x_1^* + \xi_3' x_3^*$
for some $x_3^*\in B_{\ell_\infty}$ and $\xi_1', \xi_3'\in
S_\mathbb{C}$. By (\ref{eq:1sameargument2}) and (\ref{eq2:same4}),
\begin{eqnarray*}
v(h) &=& | x^*_1(f(\lambda_0x_0)) + \eta_1 y^*(\lambda_0x_0) +
\eta_2 \|g\| +\overline{\xi_1'}\xi_3'\|y\||\\ &=&|x_1^*(f(\lambda_0
x_0))+ \eta_1 y^*(\lambda_0 x_0) + \eta_2\|g\| +
\xi_3\|y\||.\end{eqnarray*} Since $x^*(x_0)=\overline{\lambda_0}$,
we get $\xi_1' = \overline{\lambda_0}$ and
$\xi_3'=\overline{\lambda_0}\xi_3$. Hence $x^* =
\overline{\lambda_0}z^*$ and $\{x_n^*\}$ converges weak-$*$ to
$\overline{\lambda_0} z^*$. Thus $h$ is a numerical strong peak
function at $(\lambda_0 x_0, \overline{\lambda_0} z^*)$.

The proof of the last case $\mathbb{N}={\rm supp}(x_0)$ is easy and
similar to the previous case. In either case, $\|f-h\|\le
3\epsilon$. This completes the proof.
\end{proof}

\section{Applications and  Numerical boundary}

The following proposition is a numerical version of Bishop's Theorem
\cite{CLS}.

\begin{prop}
Let $A$ be a subspace of $C_b(B_X:X)$. Suppose that the set of all
numerical strong peak functions in $A$ is dense in $A$. Then the
$\tau$-closure of the set of  all numerical strong peak points for
$A$ is the numerical Shilov boundary of $A$.
\end{prop}
\begin{proof} Let \[\Gamma:=\{(z, z^*)\in \Pi(X)~:~(z, z^*)~\mbox{is a
numerical strong peak points for}~A\}.\] Notice that every
$\tau$-closed numerical boundary of $A$ contains $\Gamma$. Hence the
numerical Shilov boundary of $A$ contains all points of
$\tau$-closure of $\Gamma$. For the reverse inclusion, let $f\in A$
be fixed. Then there exists a sequence $\{f_n\}_{n=1}^{\infty}$ of
numerical strong peak functions in $A$ such that $\|f_n-f\|\to
\infty$. Hence $|v(f_n)-v(f)|\to \infty$. Let $(x_n,
x_n^*)_{n=1}^{\infty}$ be a sequence of numerical strong peak points
in $\Pi(X)$ such that $v(f_n)=|x_n^*(f(x_n))|$ for every $n$. Then
$v(f)=\lim_{n\to \infty}{|x_n^*(f_N(x_n))|}$. Note that for every
$n$,
\begin{align*}|x_n^*(f_N(x_n))|-|x_n^*((f_n-f)(x_n))|&\leq|x_n^*(f(x_n))|\\&\leq
|x_n^*(f_N(x_n))|+|x_n^*((f_n-f)(x_n))|,\end{align*} shows that
$v(f)=\lim_{n\to \infty}{|x_n^*(f(x_n))|}=\sup\{|z^*(f(z))|~:~(z,
z^*)\in \Gamma \}$. So $\Gamma$ is a numerical boundary for A. Thus
the numerical Shilov boundary of $A$ is contained in the
$\tau$-closure of $\Gamma$.
\end{proof}

\begin{prop}Let $A$ be a subspace of $C_b(B_X:X)$ which contains all
functions of the forms:
\[ x^*\otimes y, \ \ \ \ \ 1\otimes z, \ \ \ \ \ \forall x^*\in
X^*,\ \forall y, z\in X.\] Suppose that $X$ is smooth and locally
uniformly convex.

Then  $\Pi(X)$ is the set of all numerical strong peak points for
$A$ and $\Pi(X)$ is the numerical Shilov boundary of $A$.
\end{prop}
\begin{proof}
Let $(x_0, x_0^*)\in\Pi(X)$. Then $g(x) = (x^*_0(x)+1)/2$ is a
strong peak function at $x_0$. Let $h(x): = g(x)x_0\in A$. Then $h$
is a numerical strong peak function at $(x_0, x_0^*)$. So $\Pi(X)$
is the set of all numerical strong peak points for $A$. The second
assertion follows from Theorem 2.19, Proposition 3.1, and the first.
\end{proof}

The following is an application of the numerical boundary to the
density of numerical radius attaining holomorphic functions. Similar
application of the norming subset to the density of norm attaining
holomorphic functions is given in \cite{CLS}. We use the
Lindenstrauss method \cite{Li}.

\begin{thm}\label{thm:nradius}
Suppose that $X$ is a Banach space and there is a numerical boundary
$\Gamma\subset \Pi(X)$ of $A(B_X:X)$ such that for every $(x,
x^*)\in \Gamma$, $x$ is a strong peak point of $A(B_X)$. Then the
set of numerical radius attaining elements  in $A(B_X:X)$ is dense.
\end{thm}
\begin{proof}
We may assume that $\Gamma= \{(x_\alpha, x_\alpha^*)\}_\alpha$ and
$\varphi_\alpha(x_\alpha)=1$ for each $\alpha$, where each
$\varphi_\alpha$ is a strong peak function at $x_\alpha$ in
$A(B_X)$. Let $f\in A(B_X:X)$ with $\|f\|=1$ and $\epsilon$ with
$0<\epsilon<1/3$ be given. Without loss of generality, we may assume
that $v(f)>0$. We choose a monotonically decreasing sequence
$\{\epsilon_k\}$ of positive numbers so that
\begin{equation}\label{eq1} 2\sum_{i=1}^\infty \epsilon_i <\epsilon,
\ \ \ 2\sum_{i=k+1}^\infty \epsilon_i < \epsilon_k^2,\ \ \
\epsilon_k < \frac 1{10k},\ \ \ k=1,2,\ldots\end{equation}

We next choose inductively sequences $\{f_k\}_{k=1}^\infty$,
$\{(x_{\alpha_k}, x^*_{\alpha_k})\}_{k=1}^\infty$ satisfying
\begin{align}
\label{eq2}&f_1=f\\
\label{eq3}&|x_{\alpha_k}^* (f_k(x_{\alpha_k}))| \ge v(f_k)-\epsilon_k^2\\
\label{eq4}&f_{k+1}(x) = f_k(x) + \epsilon_k
\tilde{\varphi}_{\alpha_k}(x)\cdot f_k(x_{\alpha_k})\\
\label{eq9}|\tilde{\varphi}_{\alpha_k}(x)|&>1-1/k \ \ \mbox{ implies
} \ \ \|x-x_{\alpha_k}\|<1/k,
\end{align} where  $\tilde{\varphi}_{\alpha_j}$ is $\varphi^{n_j}_{\alpha_j}$ for some positive integer $n_j$.
Having chosen these sequences, we verify the following hold:
\begin{align}
\label{eq5}\|f_j - f_k\|&\le 2\sum_{i=j}^{k-1} \epsilon_i,\ \ \  \|f_k\|\le 4/3, &j<k,\ \ \  k= 2, 3, \ldots\\
\label{eq6}v(f_{k+1})& \ge v(f_k)+\epsilon_kv(f_k) - 2\epsilon_k^2,& k=1, 2,\ldots\\
\label{eq8}|\tilde{\varphi}_{\alpha_j}(x_{\alpha_k})|&>1-1/j, &
j<k,\ \ \ k=2, 3,\ldots.
\end{align}

Assertion (\ref{eq5}) is easy by using induction on $k$. By
(\ref{eq3}) and (\ref{eq4}), \begin{align*} v(f_{k+1})&\ge
|x_{\alpha_k}^*(f_{k+1} (x_{\alpha_k}))| =
|x_{\alpha_k}^*(f_k(x_{\alpha_k}))|~|1+\epsilon_k
\tilde{\varphi}_{\alpha_k}(x_{\alpha_k})| \\ &=
|x_{\alpha_k}^*(f_k(x_{\alpha_k}))|(1+\epsilon_k)\ge
(v(f_k)-\epsilon_k^2)(1+\epsilon_k)\\&\ge
v(f_k)+\epsilon_kv(f_k)-2\epsilon_k^2,\end{align*} so the relation
(\ref{eq6}) is proved.  For $j<k$, by the triangle inequality,
(\ref{eq3}) and (\ref{eq5}),  we have
\begin{align*} |x_{\alpha_k}^*(f_{j+1}(x_{\alpha_k}))|&\ge
|x_{\alpha_k}^*(f_k(x_{\alpha_k}))|-\|f_k -f_{j+1}\|\\&\ge
v(f_k)-\epsilon^2_k-2\sum_{i=j+1}^{k-1}\epsilon_j \ge
v(f_{j+1})-2\epsilon_j^2.\end{align*} Hence by (\ref{eq4}) and
(\ref{eq6}), \begin{align*}
\epsilon_j|\tilde{\varphi}_{\alpha_j}(x_{\alpha_k})| \cdot v(f_j) +
v(f_j)&\ge |x_{\alpha_k}^*(f_{j+1}(x_{\alpha_k}))|\ge v(f_{j+1})-2\epsilon^2_j\\
&\ge v(f_j) + \epsilon_j v(f_j)-4\epsilon_j^2,\end{align*} so that
\[ |\tilde{\varphi}_{\alpha_j}(x_{\alpha_k})| \ge 1-4\epsilon_j>1-1/j\] and
this proves (\ref{eq8}). Let $\hat{f}\in A(B_X:X)$ be the limit of
$\{f_k\}$ in the norm topology. By (\ref{eq1}) and (\ref{eq5}),
$\|\hat{f}-f\|=\lim_n \|f_n-f_1\|\le2\sum_{i=1}^\infty \epsilon_i\le
\epsilon$ holds. The relations (\ref{eq9}) and (\ref{eq8}) mean that
the sequence $\{x_{\alpha_k}\}$ converges to a point $\tilde{x}$,
say and  by (\ref{eq3}), we have $v(\hat{f})=\lim_n v(f_n) = \lim_n
|x_{\alpha_n}^*(f_n(x_{\alpha_n}))| =
|\tilde{x}^*\hat{f}(\tilde{x})|$, where $\tilde{x}^*$ is the
weak-$*$ limit point of $\{x_{\alpha_k}^*\}_k$ in $B_{X^*}$. Then it
is easy to see that $|\tilde{x}^*(\tilde{x})|=1$. Hence $\hat{f}$
attains its numerical radius. This concludes the proof.
\end{proof}

\begin{cor}
Let $X$ be a locally uniformly convex Banach space. Then the set of
numerical radius attaining elements in $A(B_X:X)$ is dense.
\end{cor}
\begin{proof}
Let $\Gamma=\Pi(X)$ and notice that every element in $S_X$ is a
strong peak point for $A_u(B_X)$. Indeed, if $x\in S_X$, choose
$x^*\in S_{X^*}$ so that $x^*(x)=1$. Set $f(y) = \frac{x^*(y)+1}2$
for $y\in B_X$. Then $f\in A(B_X)$ and $f(x)=1$. If $\lim_n|f(x_n)|=
1$ for some sequence $\{x_n\}$ in $B_X$, then $\lim_n x^*(x_n)=1$.
Since $|x^*(x_n) + x^*(x)|\le \|x_n + x\|\le 2$ for every $n$,
$\|x_n + x\|\to 2$ and $\|x_n -x\|\to 0$ as $n\to \infty$. By
Theorem~\ref{thm:nradius}, we get the desired result.
\end{proof}

It was shown in \cite{CHL} that if a Banach sequence space $X$ is
locally uniformly $c$-convex and order continuous, then the set of
all strong peak points for $A(B_X)$ is dense in $S_X$. Therefore,
the set of all strong peak points for $A(B_X)$ is dense in $S_X$.
For the characterizations of the local uniform $c$-convexity in
function spaces, see \cite{Lee3}.

\begin{cor}\label{cor:analyticcase}
Suppose that $X$ is a locally uniformly $c$-convex, order continuous
Banach sequence space. Then the set of numerical radius attaining
elements  in $A_u(B_X:X)$ is dense.
\end{cor}
\begin{proof}
Let $\Gamma= \{(x, x^*): x \mbox{ is a strong peak point of }
A_u(B_X)\}$. Then by \cite[Theorem~2.5]{Pal} and the remark above
the Corollary~\ref{cor:analyticcase}, $\Gamma$ is a numerical
boundary of $A_u(B_X:X)$. Hence the proof is complete by
Theorem~\ref{thm:nradius}.
\end{proof}

\section{Negative Results for Denseness of Numerical Peak Holomorphic
Functions}

It is observed in \cite{CHL} that there is no strong peak function
in $A_{wu}(B_{L_1[0,1]})$. The following
 shows that there is no numerical strong peak function in
$A_{wu}(B_{L_1[0,1]}:L_1[0,1])$.

\begin{prop}
There is no numerical strong peak function in
$A_{wu}(B_{L_1[0,1]}:L_1[0,1])$.
\end{prop}
\begin{proof}
Let $\{r_n\}_{n=1}^\infty$ be Rademacher functions on $[0,1]$. We
shall use the following basic fact observed in \cite{LT}:
\[\lim_{n\to \infty} \|x+ r_n x\|_1= \|x\|_1,\ \ \ \ \forall x\in
L_1[0,1].\] Notice also that if we let $x_n = (1+r_n)x$ for $n\ge 1$
and for some $x\in X$, then $x_n$ weakly converges to $x$.

Suppose that there is a numerical strong peak function $f\in
 A_{wu}(B_{L_1[0,1]}:L_1[0,1])$ at $(x, x^*)\in \Pi(L_1[0,1])$. Then
\[ 1= \int_0^1 x(t)x^*(t)\, dt = \int_0^1 |x(t)|\, dt\] shows that
\[ x^*(t) = \overline{{\rm sign} x(t)}, \ \ \ \ \ a.e.\  t\in {\rm supp} (x).\]

If we take $y_n = \frac {(1+r_n)x}{\|(1+r_n)x\|_1}$ for each $n\ge
1$, then $\{y_n\}_{n=1}^{\infty}$ weakly converges to $x$ and
\begin{eqnarray*} x^*(y_n) &=
\frac{1}{\|(1+r_n)x\|_1} \int_0^1 x^*(t) (1+r_N(t))x(t)\, dt \\
& = \frac{1}{\|(1+r_n)x\|_1} \int_{{\rm supp} x}  x^*(t)
(1+r_N(t))x(t)\,
dt\\
&= \frac{1}{\|(1+r_n)x\|_1} \int_{{\rm supp} x}  (1+r_N(t))|x(t)|\,
dt = 1.
\end{eqnarray*} Therefore $(y_n, x^*)\in \Pi(X)$ so $\{|x^*(f(y_n))|\}_{n=1}^{\infty}$
converges to $|x^*(f(x))|=v(f)$. However
\[ \|x-y_n \|_1 \ge \|x-(1+r_n)x\|_1 - \|y_n - (1+r_n) x \|_1\] shows
that  $ \liminf_n \|x-y_n\|_1\ge 1.$ This contradicts that $f$ is a
numerical strong peak function at $(x, x^*)$.
\end{proof}

We recall the definition of {\it property ($\beta$)} introduced by
J. Lindenstrauss \cite{Li}. It generalizes the geometric behavior of
the standard biorthogonal pairs $\{(e_n, e_n^*)\}_{n=1}^\infty$ in
$c_0$ and $\ell_\infty$. A Banach space $Y$ has {\it property
($\beta$)} with constant $\rho$ $(0\le \rho \le 1)$ if there is some
constant $0\le \rho <1$ and a subset $\{(y_i ,y_i^*):i\in I\}\subset
Y\times Y^*$ satisfying
\begin{enumerate}
\item $\|y_i\| = \|y_i^*\| = y_i^*(y_i) =1$ for all $i\in I$,
\item $|y_i(y_j)|\le \rho$ for all $i,j\in I$, $i\neq j$.
\item $\|y\| = \sup\{ |y_i^*(y)| : i\in I\}$ for every $y\in Y$.
\end{enumerate}
J. Partington \cite{P} proved that every Banach space can be
equivalently renormed to have the property ($\beta$).

We say that a Banach space $Y$ with property ($\beta$) has property
${\bf Q}$ if $\|y\| = \max\{ |y_i^*(y)|:i\in I\}$ for every $y\in
Y$. Note that $c_0$ has property ${\bf Q}$.

\begin{prop}\label{prop:counter1}
Let $E$ be a complex Banach space having $(\beta)$-and ${\bf Q}$-
properties with $\rho=0$. There are no numerical peak functions in
$A_b(B_{E}:E)$.
\end{prop}
\begin{proof}
It suffices to show the proposition when $E=c_0$. Assume that there
exists a numerical peak function $f\in A_b(B_{c_0}:c_0)$. There
exists $(x_0, x_0^*)\in \Pi(c_0)$ such that
\[v(f)=|x_0^*(f(x_0))|>|z^*(f(z))|~\mbox{for every}~(z, z^*)\in \Pi(c_0)\backslash \{(x_0,
x_0^*)\}.\] Note that $v(f)=\|f\|$ (see \cite{AK}). There exists
$n_0\in \mathbb{N}$ such that $|e_{n_0}^*(f(x_0))|=\|f(x_0)\|$.

We claim that $|e_{n_0}^*(x_0)|=1$ and $ x_0^*=\overline{{\rm
sign}(e_{n_0}^*(x_0))}e_{n_0}^*$.

Assume that $|e_{n_0}^*(x_0)|<1$. Define
\[\phi_0(\lambda)=e_{n_0}^*(f(x_0+(\lambda-e_{n_0}^*(x_0))e_{n_0}))
\ \ \ (\lambda \in \mathbb{C}, |\lambda| \le 1),\] which is a
continuous function on the closed  unit disk and holomorphic on the
open unit disk. By the Maximum Modulus Theorem, $\phi_0$ is constant
on the open unit disk. So for every $\lambda \in \mathbb{C}$ such
that $ |\lambda|=1$, we have
\[|\phi_0(\lambda)|=|\phi_0(e_{n_0}^*(x_0))| =|e_{n_0}^*(f(x_0+(\lambda-e_{n_0}^*(x_0))e_{n_0}))|
=\|f(x_0)\|=\|f\|=v(f).\] Note that
$(x_0+(\lambda-e_{n_0}^*(x_0))e_{n_0},~\overline{{\rm
sign}(\lambda)}e_{n_0}^*)\in \Pi(c_0)$. Thus
$x_0+(\lambda-e_{n_0}^*(x_0))e_{n_0}=x_0$ for every $|\lambda|=1$,
which is impossible. Thus $|e_{n_0}^*(x_0)|=1$ and
$x_0^*=\overline{{\rm sign}(e_{n_0}^*(x_0))}e_{n_0}^*$.

Choose $N\in\mathbb{N}$ such that $|e_{N}^*(x_0)|<1$. Clearly $N\neq
n_0$. Define
\[\phi_1(\lambda)=e_{n_0}^*(f(x_0+(\lambda-e_{N}^*(x_0))e_{N}))
\ \ \ (\lambda \in \mathbb{C}, |\lambda| \le 1),\] which is a
continuous function on the closed  unit disk and holomorphic on the
open unit disk. By the Maximum Modulus Theorem, $\phi_1$ is constant
on the open unit disk. For every $\lambda \in \mathbb{C}$ with $
|\lambda|=1$, we have
\[|\phi_1(\lambda)|=|\phi_1 (e_{N}^*(x_0))| =|\overline{{\rm sign}(e_{n_0}^*(x_0)}e_{n_0}^*(f(x_0+(\lambda-e_{N}^*(x_0))e_{N}))|
=v(f).\] Note that $(x_0+(\lambda-e_{N}^*(x_0))e_{N},~\overline{{\rm
sign}(e_{n_0}^*(x_0))}e_{n_0}^*)\in \Pi(c_0)$. Thus
$x_0+(\lambda-e_{N}^*(x_0))e_{N}=x_0$, which is impossible.
\end{proof}

Note that Theorem 3.2(1) in [K] implies that there exist infinitely
many numerical peak functions in $A_b(B_{C(K)}:C(K))$.

\begin{prop}Let $K$ be a compact Hausdorff space and $X=C(K)$.
Suppose that $f\in A_b(B_X: X)$ is a numerical peak function at
$(x_0, x^*_0)$ in $\Pi(X)$. Then $x_0^*=\overline{{\rm
sign}(x_0(t_0))}\delta_{t_0}$ for some $t_0\in K$ and $f(x_0)\in
C(K)$ is a peak function at $t_0$.
\end{prop}
\begin{proof}
Notice that $v(g) = \|g\|$ for every $g\in A_b(B_X: X)$
\cite[Theorem~2.8]{AK}. Because $f$ is a numerical peak function at
$(x_0, x_0^*)$ we have $v(f) = \|f\| = |x_0^*f(x_0)|$. Since
$\|x_0^*\|\le 1$, we have $\|x_0^*f(x_0)\| = \|f(x_0)\|=
|\delta_{t_0}f(x_0)|$ for some $t_0\in K$.

We claim that if $\|f(x_0)\|= |\delta_{t}f(x_0)|$, then
$|x_0(t)|=1$. Otherwise we have $|x_0(t)|<1$. Choose $y\in B_X$ such
that $y(t)=1$ and define a function
\[ \varphi(\lambda) = \delta_{t} f(x_0 + \lambda y (1-|x_0|)).\]
Then $\varphi$ is a holomorphic on the open unit disc in the complex
plane and continuous on the closed unit disc. Notice also that
$|\varphi(0)| = |\delta_{t}f(x_0)| = \|f\|$. By the maximum modulus
theorem, $\varphi$ is a constant. Choose $\lambda_0$ in
$S_\mathbb{C}$ satisfying $|x_0(t) + \lambda_0 (1-|x_0(t)|)|=1$.
Then $\varphi(\lambda_0) = |\delta_{t} f(x_0 + \lambda_0 (1-
|x_0|))| = \varphi(0) = v(f)$. Since $(x_0+ \lambda_0 (1-|x_0|),~
\overline{{\rm sign}(x_0+ \lambda_0 (1-|x_0|))}\delta_{t})$ is in
$\Pi(X)$, we get $x_0^* = \overline{{\rm sign}(x_0+ \lambda_0
(1-|x_0|))}\delta_{t}$. Now $1= |x_0^*(x_0)| = |x_0(t)|$. This is a
contradiction.

Notice that $(x_0, {\rm sign}(x_0(t_0))\delta_{t_0})$ is in $\Pi(X)$
and $v(f) = |\delta_{t_0} f(x_0)|$ shows that $x_0^* ={\rm
sign}(x_0(t_0)) \delta_{t_0}$ since $f$ is a numerical peak
function. Finally, if there is $s\in K$ such that $\|f(x_0)\| =
|\delta_s f(x_0)|$, then by claim, $(x_0, \overline{{\rm
sign}{x_0(s)}} \delta_s)\in \Pi(X)$ and $|\delta_s f(x_0)| = v(f)$.
So we get $t=s$. Hence $f(x_0)$ is a peak function at $t_0$.
\end{proof}

\begin{rem}
E. Bishop showed \cite{B} that there is a compact Hausdorff space
$K$ such that $C(K)$ has no peak functions. In that case, there is
no numerical peak function in $A_b(C(K): C(K))$.
\end{rem}

Recall that a $x$ in the unit ball  $B_X$ of a complex Banach space
$X$ is said to be a {\it complex extreme point} if whenever
$\sup_{0\le \theta\le 2\pi}  \| x+e^{i\theta} y\|\le 1$ for some
$y\in X$, we get $y=0$. It is easy to see that every extreme point
of $B_X$ is a complex extreme point of $B_X$. It is observed by
Globevnik \cite{G1} that if $f\in A_b(B_X)$ is a strong peak
function at $x_0$, then $x_0$ is a complex extreme point of $B_X$.
Otherwise, there is a nonzero $w\in X$ such that $\|x_0+\lambda
w\|\le 1$ for every $\lambda\in B_\mathbb{C}$. Hence the function
$\varphi(\lambda) = f( x_0+ \lambda w)$ is holomorphic on the
interior of $B_\mathbb{C}$ and continuous on $B_\mathbb{C}$. Notice
also that $\|\varphi\| = |\varphi(0)| = \|f\|$. By the strong
maximum modulus theorem, $\varphi$ is constant on $B_\mathbb{C}$.
Hence $\|f\| = |f(x_0)|= |\varphi(0)| = |\varphi(1)| = |f(x_0+w)|$.
So $|f(x_0)\| = |f(x_0+w)|$ and $x_0 = x_0 +w$ because $f$ is a peak
function at $x_0$. Hence $w=0$, which is a contradiction to the
assumption $w\neq 0$.

We denote by ${\rm ext}_\mathbb{C}(B_X)$ the set of all complex
extreme points of $B_X$ and by ${\rm ext}(B_X)$ the set of all
extreme points of $B_X$. Notice that an element $f\in {\rm
ext}(B_{C(K)})$ on a compact Hausdorff space $K$ and only if $|f|=1$
on $K$. So it is easy to check that $f\in B_{C(K)}$ is a complex
extreme point of $B_{C(K)}$ if and only if $f$ is a complex extreme
point of $B_{C(K)}$.

\begin{thm}Let $K$ be a compact Hausdorff space such that the
closed unit ball of $C(K)^*$ is the closed convex hull of
$extB_{C(K)^*}$. Then $f$ is a numerical peak function in
$A_b(B_{C(K)}:C(K))$ if and only if there exist unique $x_0\in
extB_{C(K)}$ and $t_0\in K$ such that

(a) $v(f)=\|f\|=\|f(x_0)\|>\|f(x)\|$ for every $x\in B_{C(K)}$ with
$x\neq x_0$;

(b) $v(f)=\|f\|=|\delta_{t_0}(f(x_0))|>|\delta_{t}(f(x_0))|$ for
every $t\in K$ with $t\neq t_0$.
\end{thm}

\begin{proof} Let $X= C(K)$. Notice that $v(g) = \|g\|$ for every $g\in A_b(B_{X}: X)$

($\Rightarrow$): Suppose that $f$ is a numerical peak function for
$A_b(B_{X}:X)$. Then there exists a unique $(x_0,x_0^*)\in\Pi(X)$
such that
\[|x_0^*(f(x_0))|=v(f)=\|f\|.\]

claim: $f$ is a norm peak function at $x_0$.

Suppose that $\|f(a)\|=\|f\|$ for some $a\in B_X$. Then there is
$s\in K$ such that $|f(a)(s)|=\|f\|$. We shall show that $|a(s)|=1$.
Suppose on the contrary that $|a(s)|<1$. In case that $s$ is an
isolated point, consider the function $\varphi:B_\mathbb{C}\to
\mathbb{C}$ defined by $\varphi(z) = f(a- a(s)\chi_{\{s\}} +
z\chi_{\{s\}})(s)$, where $\chi_A$ is a characteristic function on
$A$. Then $\varphi\in A(B_\mathbb{C})$ and attains its maximum at
the interior point $a(s)$ of $B_\mathbb{C}$. So the maximum modulus
theorem shows that $\varphi$ is a constant function. Notice that
\begin{align*}\|f\|=|\varphi(1)|&=|f(a- a(s)\chi_{\{s\}} + 1\chi_{\{s\}})(s)|
\\&= |f(a- a(s)\chi_{\{s\}}
-1\chi_{\{s\}})(s)|=|\varphi(-1)|.\end{align*} This means that
\[\|f\|=|\delta_{s}(f(a- a(s)\chi_{\{s\}} + \chi_{\{s\}}))| =
|-\delta_{s}(f(a- a(s)\chi_{\{s\}} -1\chi_{\{s\}}))|.\] Since $(a-
a(s)\chi_{\{s\}} + \chi_{\{s\}}, \delta_{s})$ and $(a-
a(s)\chi_{\{s\}} - \chi_{\{s\}}, -\delta_{s})$ are two different
elements in $\Pi(X)$, it is a contradiction to that $f$ is a
numerical peak function.

For the other case, suppose that $s$ is not an isolated point. So
the set $A:=\{ t\in K: |a(t)|<1\}$ contains at least one different
point other than $s$. So we can choose two $\varphi, \psi\in C(K)$
such that $\varphi(s)=\psi(s)=1$, $\varphi|_A\neq \psi|_A$ and
$|\varphi|\le 1$, $|\psi|\le 1$. Now choose a complex number $z_0\in
S_\mathbb{C}$ such that $|a(s) + z_0(1-|a(s)|)|=1$ and let
$w_1:=a(s)+ z_0(1-|a(s)|)$.

Then if we consider the function $z\mapsto f(a(\cdot)+
z_0\varphi(\cdot)(1-|a(\cdot)|))(s)$, it belongs to
$A(B_\mathbb{C})$ and attains its maximum $\|f\|$ at 0. Hence it is
a constant function. So $\|f\|=|\overline{w_1}\delta_s(f(a(\cdot)+
z_0\varphi(\cdot)(1-|a(\cdot)|)))|$ and $(a(\cdot)+
z_0\varphi(\cdot)(1-|a(\cdot)|), \overline{w_1}\delta_s)$ is in
$\Pi(X)$. Hence $x_0=a(\cdot)+ z_0\varphi(\cdot)(1-|a(\cdot)|)$
because $f$ is a numerical peak function at $(x_0, x_0^*)$.
Similarly, we have $x_0=a(\cdot)+ z_0\psi(\cdot)(1-|a(\cdot)|)$.
Then $\varphi(t) = \psi(t)$ if $|a(t)|<1$. This is a contradiction
to that $\varphi|_A\neq \psi|_A$.

Therefore we show that $|a(s)|=1$. Then $(a,
\overline{a(s)}\delta_s)\in \Pi(X)$ and $\|f\|= |f(a)(s)|$ shows
that $a=x_0$. This proves that $f$ is a peak function at $x_0$. Then
$x_0$ is a complex extreme point of $B_X$.  Thus $x_0\in {\rm
ext}(B_{X})={\rm ext}_\mathbb{C}(B_X)$.

Since $f(x_0)\in C(K),$ there is $t_0\in K$ such that
\[|{\rm sign}(x_0(t_0))\delta_{t_0}(f(x_0))|=|f(x_0)(t_0)|=v(f)=\|f\|=\|f(x_0)\|.\]

claim: $v(f)=\|f\|=|\delta_{t_0}(f(x_0))|>|\delta_{t}(f(x_0))|$ for
every $t\in K$ with $t\neq t_0$.

Since $x_0\in ext B_X$, $|x_0(s)|=1$ for every $s\in K$. Hence if
$t\in K$ and $t\neq t_0$, then $(x_0, {\rm sign}(x_0(t))
\delta_t)\in \Pi(X)$. Notice that  $f$ is a numerical peak function
at $(x_0, {\rm sign}(x_0(t_0))\delta_{t_0})$. Hence
\[ v(f)=\|f\|=|\delta_{t_0}(f(x_0))|>|\delta_{t}(f(x_0))|\] because
$\delta_t \neq \delta_{t_0}$ by the Urysohn Lemma.

($\Leftarrow$): Let $x_0\in extB_{X}$ and $t_0\in K$ satisfying the
conditions (a) and (b).

claim: $f$ is a numerical peak function at $(x_0, \overline{{\rm
sign}(x_0(t_0))}\delta_{t_0})$

Let *$(y, y^*)\in \Pi(X)$ be such that $|y^*(f(y))|=v(f)=\|f\|.$
Since $\|f\|=\|f(y)\|\geq |y^*(f(y))|=v(f)=\|f\|,$ by condition (a),
we have $y=x_0$. The hypothesis
$B_{X^*}=\overline{co}(ext_{B_{X^*}})$ implies that there exist
sequences of nonzero complex numbers $\{\lambda_n\}$ with
$\sum_{n=1}^{\infty}{|\lambda_n|}\leq 1$ and $\{t_n\}$ in $K$ such
that $y^*=\sum_{n=1}^{\infty}{\lambda_n\delta_{t_n}}$. We claim that
$t_n=t_0$ for every $n$. Indeed, assume that $t_{n_0}\neq t_0$ for
every $n_0$. It follows that
\begin{eqnarray*}
\|f\|&=&|y^*(f(x_0))|\\
&\leq&\sum_{n=1}^{\infty}{|\lambda_n| ~|f(x_0)(t_n)|}\\
&=&|\lambda_{n_0}| ~|\delta_{t_{n_0}}(f(x_0))|+\sum_{n\neq n_0}{|\lambda_n|~\delta_{t_{n}}(f(x_0))|}\\
&<&|\lambda_{n_0}| ~\|f\|+\sum_{n\neq n_0}{|\lambda_n| ~|f(x_0)(t_n)|}~~(\mbox{by condition (b)})\\
&\leq &\sum_{n=1}^{\infty}{|\lambda_n|}\|f\|=\|f\|,
\end{eqnarray*}
which is a contradiction. Thus,
$y^*=\sum_{n=1}^{\infty}{\lambda_n}\delta_{t_0}$. Since $(x_0,
y^*)\in \Pi(X)$, we have $y^*=\overline{{\rm
sign}(x_0(t_0))}\delta_{t_0}$, so we have shown that $f$ is a
numerical peak function at $(x_0, \overline{{\rm
sign}(x_0(t_0))}\delta_{t_0}).$
\end{proof}

\begin{rem}
Let $K$ be a completely regular Hausdorff space. The $C_b(K)$ is
isometrically isomorphic with $C(\beta K)$, where $\beta K$ is the
Stone-\u{C}ech compactification of $K$. Therefore $f$ is a numerical
peak function in $A_b(C_b(K):C_b(K))$ if and only if there exist
unique $x_0\in extB_{C_b(K)}$ and $x_0^*\in {\rm ext} B_{C_b(K)^*}$
such that
\begin{enumerate}
\item[(a)] $v(f)=\|f\|=\|f(x_0)\|>\|f(x)\|$ for every $x\in B_{C_b(K)}$
with $x\neq x_0$;

\item[(b)] $v(f)=\|f\|=|x_0^*(f(x_0))|>|y^*(f(x_0))|$ for every $y^*\in
{\rm ext} B_{C_b(K)^*}$ with $y^*\neq x_0^*$.
\end{enumerate}
\end{rem}

\begin{rem}
In general, it is not true that if $f$ is a peak function for
$A_b(B_{C(K)}:{C(K)})$, then $f$ is a numerical peak function.

{\em Indeed, $K$ is finite with more than two elements. Then $C(K)
=\ell_\infty^n$ for some $n> 1$. Hence there is a peak function
$h\in A_u(B_{C(K)})$ at $x_0$ with $\|h\|=1$ since $C(K)$ is finite
dimensional. Given two distinct point $t_0$ and $t_1$ in $K$, choose
a function $g\in C(K)$ such that $\|g\|=1 = g(t_0) = g(t_1)$. Hence
$x_0\in ext B_{C(K)}$ and $|x_0|=1$ on $K$. Hence if we define $f:
B_{C(K)}\to C(K)$ by $f(x): = h(x)g$ for each $x\in B_{C(K)}$, then
$\|f\|= \|f(x_0)\|=|h(x_0)|> |h(x)|=\|f(x)\|$ for any $x\in
B_{C(K)}$ with $x\neq x_0$. Hence $f$ is a peak function on
$A_u(B_{C(K)}:C(K))$. However $(x_0, {\rm
sign}(x_0(t_0))\delta_{t_0})$ and $(x_0, {\rm
sign}(x_0(t_1))\delta_{t_1})$ are in $\Pi(X)$ and it is clear that
\[ |\delta_{t_0}(f(x_0))|=1 = |\delta_{t_1}(f(x_0))|.\] Therefore $f$
is not a numerical peak function.}
\end{rem}

\begin{prop}\label{prop:nemericalpeakchar}
Suppose that $X$ is finite dimensional and $N(A_u(B_X:X))=1$. Then
$f$ is a numerical peak function in $A_u(B_X:X)$ if and only if
there exist $x_0\in {\rm ext}_\mathbb{C}(B_X)$ and $x_0^*\in {\rm
ext}(B_{X^*})$ such that
\begin{enumerate}
\item[(a)] $v(f)=\|f\|=\|f(x_0)\|>\|f(x)\|$ for every $x\in {\rm
ext}_\mathbb{C}(B_X)$ with $x\neq x_0$;

\item[(b)] $v(f)=|x_0^*(f(x_0))|>|y^*(f(x_0))|$ for every $y^*\in
B_{X^*}$ with $y^*\neq x^*_0$ and $y^*(x_0)=1$.

\item[(c)] $v(f)>|x^*_0(f(y))|$ for every $y\in B_{X}$ with $y\neq x_0$
and $|x^*_0(y)|=1$.\end{enumerate}
\end{prop}
\begin{proof}
Suppose that $f$ is a numerical peak function in $A_u(B_X:X)$ at
$(x_0, x^*_0)$. So $\|f\|=v(f) =|x^*_0(f(x_0))|=\alpha
x^*_0(f(x_0))$ for some $\alpha \in S_\mathbb{C}$.

Then the set $T=\{ x^*\in B_{X^*}: x^*(x_0)=1, \ \alpha
x^*(f(x_0))=\|f\|\}$ is a nonempty weak-$*$ compact subset of
$B_{X^*}$. Hence $T$ has an extreme point $y^*$. Since $T$ is an
extremal subset of $B_{X^*}$, $y^*$ is an extreme point of
$B_{X^*}$. Let $\varphi(x):= y^*(f(x))$ be the function in
$A_u(B_X)$. Then by the Bishop theorem \cite{B},
\[\|\varphi\|= \max_{x\in \rho A_u(B_X)}
|\varphi(x)|,\] where $\rho A_u(B_X)$ is the set of all peak points
of $A_u(B_X)$. So there is $x_1\in \rho A_u(B_X)$ such that
$\|f\|=\|\varphi\|=|\varphi(x_1)|=|y^*(f(x_1))|$. Hence $x_1\in {\rm
ext}_\mathbb{C}(B_X)$.

Since $N(A_u(B_X:X))=1$, $|y^*(x_1)|=1$ by Corollary~2.10 in
\cite{Lee4}. Notice that $|y^*(f(x_1))|=\|f\|=v(f)$. Since $(x_1,
\overline{{\rm sign}(y^*(x_1))}y^*)\in \Pi(X)$ and $f$ is a
numerical peak function, we get  $x_1 = x_0$ and $x_0^* =
\overline{{\rm sign}(y^*(x_1))}y^*$. Hence $x_0\in {\rm
ext}_\mathbb{C}(B_X)$ and $x_0^*\in {\rm ext}(B_{X^*})$. Since $f$
is a numerical peak function at $(x_0, x_0^*)$, both (b) and (c)
hold clearly.

Fix $x_1\in {\rm ext}_\mathbb{C}(B_X)$ with $x_1\neq x_0$. Then the
set $S=\{ x^*\in B_{X^*}: x^*f(x_1)= \|f(x_1)\|\}$ is a nonempty
weak-$*$ compact subset of $B_{X^*}$ and an extremal subset of
$B_{X^*}$. Hence there is $y^*\in {\rm ext}(B_{X^*})$ such that
$y^*(f(x_1))=\|f(x_1)\|$. Hence $(x_1, \overline{{\rm
sign}(y^*(x_1))}y^*)\in \Pi(X)$ and
$\|f\|=\|f(x_0)\|=|x_0^*(f(x_0))|=v(f) > |y^*(f(x_1))|=\|f(x_1)\|$
because $f$ is a numerical peak function at $(x_0, x_0^*)$. This
shows that (a) holds.

Conversely, suppose that  $f$ in $A_u(B_X:X)$ satisfies both (a) and
(b). By Corollary~2.10 in \cite{Lee4}, $(x_0,
\overline{x_0^*(x_0)}x_0^*)\in \Pi(X)$ and $v(f) =|x^*_0(f(x_0))|$.
By (b), $x^*_0(x_0)=1$. Suppose that $(w,w^*)\in \Pi(X)$ such that
$v(f) = |w^*(f(w))|$. Choose $\gamma\in S_\mathbb{C}$ such that
$|w^*(f(w))|=\gamma w^*(f(w))$. Then $v(f) = \|f\| = \|f(w)\|$. So
if we let
\[R=\{x^*\in B_{X^*}: \gamma x^*(f(w)) = \|f(w)\|, \ x^*(w)=1\},\]
then $R$ is a nonempty weak-$*$ compact subset of $B_{X^*}$ and an
extremal subset of $B_{X^*}$. Hence there is $t^*\in {\rm
ext}(B_{X^*})$ such that $\gamma t^*(f(w))=\|f(w)\|$ and $t^*(w)=1$.
If we consider the function $\psi(x) = t^*(f(x))$ on $B_X$, $\psi\in
A_u(B_X)$. By the Bishop theorem again, there is $t\in {\rm
ext}_\mathbb{C}(B_X)$ such that $\|\psi\|=|\psi(t)|=|t^*(f(t))|\ge
|\psi(w)|=\|f(w)\|$. So $\|f(x_0)\|=v(f)= \|f(w)\|\le \|f(t)\|$.
Hence $t=x_0$ by (a). Then $|t^*(f(x_0))|=\|f\|$. By Corollary~2.10
in \cite{Lee4}, $|t^*(x_0)|=1$. So $(x_0, \overline{{\rm
sign}(t^*(x_0))}t^*)\in \Pi(X)$. Then (b) shows that $\overline{{\rm
sign}(t^*(x_0))}t^* = x_0^*$. Both $|x_0^*(w)|=1$ and (c) imply that
$w=x_0$. Then by (b), $w^*=t^*=x^*$. This shows that $f$ is a
numerical peak function at $(x_0, x_0^*)$.
\end{proof}

\begin{prop}Let $\Omega$ be a locally compact
Hausdorff space with more than 2 elements. Then there are no
numerical peak functions in $A_b(B_{C_0(\Omega)}:C_0(\Omega))$.
\end{prop}

\begin{proof} Otherwise. There exists a numerical peak function
$f\in A_b(B_{C_0(\Omega)}:C_0(\Omega))$. There exists $(x_0,
x_0^*)\in \Pi(C_0(\Omega))$ such that
\[v(f)=|x_0^*(f(x_0))|>|z^*(f(z))|~\mbox{for every}~(z, z^*)\in \Pi(C_0(\Omega))\backslash \{(x_0,
x_0^*)\}.\] Note that $v(f)=\|f\|$. There exists $t_0\in \Omega$
such that $|\delta_{t_0}(f(x_0))|=\|f(x_0)\|$.

claim: $|x_0(t_0)|=1$ and $x^*_0=\overline{{\rm
sign}(x_0(t_0))}\delta_{t_0}$

Assume that $|x_0(t_0)|<1$. Let $y_0\in B_{C(K)}$ such that
$y_0(t_0)=1$.

Define
\[\psi_0(\lambda)=\delta_{t_0}(f(x_0+\lambda y_0(1-|x_0|)))~~ (\lambda \in \mathbb{C}, |\lambda| \leq 1),\] which is a
continuous function on the closed  unit disk and holomorphic on the
open unit disk. By the Maximum Modulus Theorem, $\psi_0\equiv
\psi_0(0)$ on the open unit disk. Choose $\lambda_0 \in \mathbb{C}$
such that $|\lambda_0|=1$ and
$|x_0(t_0)+\lambda_0(1-|x_0(t_0)|)|=1$. Let $z_0:=x_0+\lambda_0
y_0(1-|x_0|))\in B_{C_0(\Omega)}$. Note that $(z_0,~\overline{{\rm
sign}(z_0(t_0))}\delta_{t_0})\in \Pi(C_0(\Omega))$ and
\[|\overline{{\rm sign}(z_0(t_0))}\delta_{t_0}(f(z_0))|=\psi_0(\lambda_0)=v(f).\]
We must have $z_0=x_0$, which is impossible. Thus $|x_0(t_0)|=1$ and
$x^*_0=\overline{{\rm sign}(x_0(t_0))}\delta_{t_0}$.

Since $x_0\in C_0(\Omega)$, there exists $t_1\in \Omega$ such that
$|x_0(t_1)|<1$. Clearly $t_0\neq t_1$. Let $y_1\in B_{C_0(\Omega)}$
such that $y_1(t_1)=1$.

Define
\[\psi_1(\lambda)=\delta_{t_0}(f(x_0+\lambda y_1(1-|x_0|)))~~ (\lambda \in \mathbb{C}, |\lambda| \leq 1),\] which is a
continuous function on the closed  unit disk and holomorphic on the
open unit disk. By the Maximum Modulus Theorem, $\psi_1\equiv
\psi_1(0)$ on the open unit disk. Choose $\lambda_1 \in \mathbb{C}$
such that $|\lambda_1|=1$ and
$|x_0(t_1)+\lambda_1(1-|x_0(t_1)|)|=1$. Let $z_1:=x_0+\lambda_1
y_1(1-|x_0|))\in B_{C_0(\Omega)}$. Note that $(z_1,~\overline{{\rm
sign}(z_1(t_0))}\delta_{t_0})\in \Pi(C_0(\Omega))$ and
\[|\overline{{\rm sign}(z_1(t_0))}\delta_{t_0}(f(z_0))|=\psi_1(\lambda_1)=v(f).\]
We must have $z_1=x_0$, which is a contradiction because
$z_1(t_1)\neq x_0(t_1).$
\end{proof}

\begin{prop}Let $K$ be an infinite compact Hausdorff
space. Then there are no numerical strong peak functions in
$A_b(B_{C(K)}:C(K))$.
\end{prop}

\begin{proof}
It follows from the proof of Theorem 3.2(2) in \cite{K1}.
\end{proof}



\end{document}